\documentclass{amsart}
    \usepackage{dynkin-diagrams}
\usepackage{amsmath,amssymb,amsfonts,amsthm,tikz-cd}
\usepackage{color}
\newtheorem{theorem}{Theorem}[section]
\newtheorem{proposition}[theorem]{Proposition}
\newtheorem{lemma}[theorem]{Lemma}
\newtheorem{corollary}[theorem]{Corollary}
\newtheorem{definition}[theorem]{Definition}
\theoremstyle{definition}
\newtheorem{remark}{Remark}[theorem]
\newtheorem{example}[theorem]{Example}
\numberwithin{equation}{section}

\begin{document}
\title{An abstract approach to algebras of braids and ties}

\author[R. Fasano ]{Riccardo Fasano}
\address{Viale Giuseppe Garibaldi 130, 00069 Trevignano Romano (RM), ITALY}
\email{riccardo.fasano98@gmail.com}
\author[D. Fiorenza]{Domenico Fiorenza}
\address{Dipartimento di Matematica G. Castelnuovo\\Sapienza Universit\`a di Roma\\Piazzale Aldo Moro 2, 00185 Roma, ITALY}
\email{fiorenza@mat.uniroma1.it} 
\author[P. Papi]{Paolo Papi}
\address{Dipartimento di Matematica G. Castelnuovo\\Sapienza Universit\`a di Roma\\Piazzale Aldo Moro 2, 00185 Roma, ITALY}
\email{papi@mat.uniroma1.it}
\begin{abstract} Generalizing work of Marin ({\it
Artin groups and Yokonuma-Hecke algebras.}
Int. Math. Res. Not. IMRN 2018, no. 13, 4022--4062), we construct in a unified way all the ``braids and ties'' algebras available in literature and new ones.\end{abstract}

\maketitle
\section{Introduction}
In the paper \cite{AY}  Aicardi and Juyumaya studied an algebra $\mathcal E_n$ defined by  generators and relations; the relations were devised  by an  abstracting procedure of a non-standard presentation for the Yokonuma-Hecke algebra
\cite{JK}, finalized to build a Markov trace on it \cite{JJ}. They also provided a diagramatic interpretation of this algebra in terms of braid diagrams on $n$ strands, by adding ``ties'' which can freely move along the strands.\par In subsequent papers, the algebra
$\mathcal E_n$, named the {\it braids and ties algebra}, has been extensively studied, e.g.  in connection with link invariants \cite{AY, AY3} or from a representation theoretic point of view \cite{RH, B}.
From its  algebraic presentation the algebra $\mathcal E_n$ can be viewed as a ``type $A$'' object, and indeed (inequivalent) generalizations to type $B$ have been introduced in \cite{F}, \cite{marin}.
The paper \cite{marin} has been another source of inspiration for the present paper. Marin introduced an  extension $\mathcal C_W$  of the Iwahori-Hecke algebra of a Coxeter system $(W,S)$
and built up a family of generically surjective morphisms of $k[B_W]\to \mathcal C_W$ (here $B_W$ is the Artin braid group attached to $W$). When $W$ is of type  $A_{n-1}$, it turns out that $\mathcal C_W\cong \mathcal E_n$. 
\par
In the present paper we proceed along the lines of \cite{marin} at a more  abstract level. The main outcome  is that the construction we outline in the next paragraphs produces in a unified way all the ``braids and ties'' algebras available in literature and new ones. \par
We start  introducing  the category  of Marin data, whose objects are   quadruples $((W,S),A,\rho,a)$ consisting of a Coxeter system $(W,S)$, of
an action of $W$ on the commutative $k$-algebra $A$,
and of a $W$-equivariant map
$
  a\colon  S \rightarrow A^{\times}.
$
Starting from  such an object one can construct a ring $\mathcal{E}^{(W,S)}(A,\rho,a)$ as the quotient of $A \rtimes k[B_W]$ by the ideal generated by Hecke type relations on $g_s,\,s\in S$ (here $w\mapsto g_w$ is the standard set-theoretical section of  the projection $B_W\to W$). 

The main result, stated in Corollary \ref{cor:free}, generalizes \cite[Theorem 3.4]{marin}: if $A$ is free over $k$, then the algebra
$\mathcal{E}^{(W,S)}(A,\rho,a)$ is freely generated as an $A$-module by the $g_w,\,w\in W$ if and only if 
$ (1 - a(s)) (1-\rho(s))= 0 $ in $\mathrm{End}_k(A)$.
\par
To apply the previous constructions to instances nearer to diagram algebras and at the same time to recover Marin's algebra $\mathcal C_W$, inspired by his notion of  {\it meaningful quotients} \cite[\S 3.4]{marin}, we introduce the notion of Marin-Iwahori-Hecke monoid, as a triple $(E,\rho,e)$ consisting of a commutative monoid $E$, a $W$-action $\rho$ on the monoid $E$, and a $W$-equivariant map $
e\colon S\to \mathrm{Idempotents}(E)$ satisfying  relation $e(s)(1-\rho(s))=0\,\forall\,s\in S$,  which guarantees that $\mathcal{E}^{(W,S)}(R[E],\rho,\hat{e})$ is free over the monoid algebra $R[E]$.  Specializing to the monoid of root subsystem $\Sigma(\Phi)$ of a root system $\Phi$, acted by $W$ by conjugation,  one recovers Marin's algebras $\mathcal C_W$ as well as  a  surjective algebra map  $\mathcal{E}^{(W,S)}(R[\Sigma(\Phi)],\rho,\hat{e})\to H_R{(W,S)}(u_s)$ to the Iwahori-Hecke algebra of $(W,S)$. 
\par Let now  $(W,S)$,  $(U,T)$ be Coxeter systems.  A Juyumaya map (cf. Definition \ref{def:j-map}), is a  pair  $(\phi,e)$ consisting of a group homomorphism $\phi : W \rightarrow U$ and a $W$-equivariant map  $e\colon  S \rightarrow \Sigma(\Phi_U)$ satisfying condition \eqref{JM}.
We can  then prove the following result, which appears as Theorem \ref{prop:juyumaya} in the body of the paper and to which  we refer for undefined notation.
\begin{theorem}
Let $(W,S)$ and $(U,T)$ be Coxeter systems and let $(\phi,e)$ be a Juyumaya map between them. Let $\rho_\phi$ denote  the $W$-action on $\Sigma(\Phi_U)$ induced by $\phi\colon W\to U$ and by the canonical action of $U$ on $\Phi_U$. Then $(\Sigma(\Phi_U)_{[e]},\rho_\phi,e)$ is a Marin-Iwahori-Hecke monoid. \\ It follows that $\mathcal{E}^{(W,S)}(R[\Sigma(\Phi_U)_{[e]}],\rho_\phi,\hat{e})$ is a free $R[\Sigma(\Phi_U)_{[e]}]$-module with basis $\{g_w | w\in W\}$. Moreover, the augmentation map $\epsilon\colon R[\Sigma(\Phi_U)_{[e]}]\to R$ induces a surjective algebra homomorphism \[
\mathcal{E}^{(W,S)}(R[\Sigma(\Phi_U)_{[e]}],\rho_\phi,\hat{e})\to H_R{(W,S)}(u_s).
\]
\end{theorem}
We then proceed, through Lemmas \ref{lemma:to-produce-triples}-\ref{lem:composition}, to provide tools  yielding  examples of Juyumaya maps. The examples are displayed in Section 3; although we proceed in algebraic terms, our treatment has been guided by the geometric point of view of \cite{A}.
In the final Section 4 we show how to produce  from our general construction several  braids and ties algebras; in general if one takes the Juyumaya  map $\phi$ to be the identity of a Coxeter system $(W,S)$ one recovers Marin's $\mathcal C_W$ algebras, and in particular the braids and ties algebra of type $A$ \cite{AY} and Marin's  braids and ties algebra of type $B$.
Non trivial choiches of $\phi$ give rise to  Flores'  braids and ties algebra of type $B$ \cite{F}, to a braids and ties algebra of type $D$ (possibly new) and more examples arising from affine Weyl groups. A few directions for future research are discussed in the final section.\par
\vskip5pt
\noindent{\sl Notational conventions.} Sometimes, for better readability of formulas, we shall denote the image of an element $x$ through a map $f$ by $f_x$ rather than by $f(x)$. For action $\rho$ of a group $G$ on a set $X$ we shall write $g.x$ to mean $\rho_g(x)$.\\
We denote by $k$ a commutative unital ring that we shall use as a ground ring.

\section{Marin Algebras}

In this section we will consider a fixed Coxeter system $(W,S)$ and a commutative $k$-algebra $A$, where $k$ is a commutative ring. We denote by $\rho_{\mathrm{geom}}\colon W\to O(V_W)$ the standard geometric representation of $(W,S)$ (see e.g. \cite[Section 5.3]{H}). When convenient, we use the same symbol $s$ to denote the basis element of $V_W$ corresponding to the generator $s$ of $(W,S)$. We also denote by $\sigma_v\colon V_W\to V_W$ the reflection associated with a unit vector $v$, so that $\rho_{\mathrm{geom}}(s)=\sigma_s$.
 Finally, we denote by $\Phi_W$ (or simply by $\Phi$ when no confusion is possible) the root system of $(W,S)$, by $\Phi^+$ the subset of positive roots with respect to the basis $S$ of $V_W$, and by $\langle \alpha\rangle $ the root subsystem $\{\alpha,-\alpha\}$ of $\Phi$ (cf. \cite[\S 5.3]{H}).
 (According to our notational convention, if $\alpha$ is a simple root attached to $s\in S$,  we write interchangeably $\langle \alpha\rangle$ or $\langle s \rangle$).
%
 \begin{definition} Let $G$ be a group and let $X$ and $Y$ be $G$-sets,  with  left $G$-action  denoted by $(g,x)\mapsto g. x$. Let $X_0\subseteq X$ and $Y_0\subseteq Y$ be subsets of $X$ and $Y$, not necessarily $G$-subsets, and let $f\colon X_0\to Y_0$ be a map. We will say that $f$ is $G$-equivariant if for any triple $(x_1,x_2,g)$ in $X_0\times X_0\times G$ with $g. x_1=x_2$ we have $g. f(x_1)=f(x_2)$.
 \end{definition}

\begin{example}\label{ex:same-stabilizer}
Given a Coxeter system $(W,S)$, both the group $W$, with the $W$-action given by conjugation, and the set $\Sigma(\Phi_W)$ of root subsystems of $\Phi_W$ with the $W$-action induced by the standard geometric representation of $(W,S)$ are $W$-sets. The map
$
S\to \Sigma(\Phi_W),\,
s\mapsto \langle s\rangle,
$,
where $\langle s\rangle$ is the root subsystem generated by $s$, i.e. $\{s,-s\}$, is $W$-equivariant. Indeed, if a triple of elemens $s_1,s_2,w$ in $S\times S\times W$ satisfies $s_2=ws_1w^{-1}$, then we have $\sigma_{s_2}=\rho_{\mathrm{geom}}(w)\sigma_{s_1}\rho_{\mathrm{geom}}(w)^{-1}=\sigma_{w. s_1}$, and so $\langle s_2\rangle =\langle w. s_1\rangle=w. \langle s_1\rangle$. Moreover, since standard geometric representation is faithful, this shows that  $s_2=ws_1w^{-1}$ if and only if $\langle s_2\rangle =w. \langle s_1\rangle$. In particular, for any $s\in S$ the stabilizer of $\langle s \rangle$ for the $W$-action on root subsystems of $\Phi$ coincides with the stabilizer of $s$ for the $W$-action on itself by conjugation, i.e., with the centralizer $C_W(s)$ of $s$ in $W$. 
\end{example}

\begin{definition}
A {\it Marin datum} is  a quadruple $((W,S),A,\rho,a)$ consisting of a Coxeter system $(W,S)$, of
an action of $W$ on the commutative $k$-algebra $A$,
$
\rho \colon  W \rightarrow \mathrm{Aut}_{k\text{-}\mathrm{alg}}(A),
$
and of a $W$-equivariant map
$
  a\colon  S \rightarrow A^{\times},\,a(s)= a_s
$
where
 $A^{\times}$ is the group of units of $A$. \end{definition}
 Marin data naturally form a category: a morphism $((W,S),A,\rho,a)\to ((U,T),$ $B,\eta,b)$ is a pair $(\phi,\epsilon)$ consisting of a morphism of Coxeter systems $\phi\colon (W,S)\to (U,T)$ and of a homomorphism of $k$-algebras $f\colon A\to B$ such that $\epsilon$ is $W$-equivariant with respect to the $W$-action $\rho$ on $A$ and the $W$-action $\eta\circ\phi$ on $B$, and such that $\epsilon \circ a=b\circ \phi$, where we use that $\phi$ maps $S$ to $T$ by definition of morphism of Coxeter systems.

\begin{remark}
Since every element $s$ of $S$ is a fixed point for the conjugacy action of $s$ on $W$, the equivariancy of $a$ implies $s. a_s=a_s$ for any $s\in S$.
\end{remark}

 Let $B_W$ be the Artin braid group attached to the Coxeter system $(W,S)$. Recall that there is a canonical (set-theoretical) section $w\mapsto g_w$ for the projection $\pi:B_W\to W$.
Let $k[B_W]$ be the group ring of  $B_W$ over $k$. The $k$-algebra $A$ is a module for $k[B_W]$ via $\pi$, so we can form the semidirect product $A \rtimes k[B_W]$, i.e., the algebra generated by the elements $c$ of $A$ and $g_w$ of $B_W$ with the relations
\[
g_w c g_{w}^{-1}=w. c,
\]
i.e., equivalently,
$
g_w c =(w. c)g_{w}.
$
\begin{definition}
    The  \textit{Marin ring} $\mathcal{E}^{(W,S)}(A,\rho,a)$ is  the quotient of $A \rtimes k[B_W]$ by the ideal generated by the elements
    \begin{equation}\label{H1}
        g_s^2 - a_s - (1 - a_s)g_s \quad \text{for all }s\in S. 
    \end{equation}
    \end{definition}

\begin{remark}\label{rem:invertibility}
Even without assuming the invertibility of $g_s$ from the group structure of $B_W$, the equation $g_s^2 = a_s + (1 - a_s)g_s$ together with the invertibility of $a_s$ in $A$ implies the invertibility of $g_s$ in $\mathcal{E}^{(W,S)}(A,\rho,a)$. Indeed, the element
${a_s}^{-1}g_s-a_s^{-1}(1-a_s)$ is the inverse of $g_s$ in $\mathcal{E}^{(W,S)}(A,\rho,a)$.
\end{remark}
    \begin{remark}\label{2var} Mimicking the construction of generic Hecke algebras given e.g. in \cite[\S 7.1]{H}, one can define a more general Marin ring 
    \textit{Marin ring} $\mathcal{E}^{(W,S)}(A,\rho,b,c)$, where $b,c: S\to A^\times$ are $W$-equivariant maps and relation \eqref{H1} is replaced by 
        \begin{equation}\label{H2}
        g_s^2 - b_s - (1 - c_s)g_s \quad \text{for all }s\in S. 
    \end{equation}
    As for  generic Hecke algebras, one reduces \eqref{H2} to \eqref{H1} a suitable rescaling $g_s\rightsquigarrow  \lambda_sg_s$. See Remark \ref{2p} for an example of this rescaling.
\end{remark}
    
\begin{remark}\label{rem:presentation}
When $A$ is a free $k$-module, recalling the relations defining the Artin ring $B_W$ one obtains a presentation of $\mathcal{E}^{(W,S)}(A,\rho,a)$ as follows: \break $\mathcal{E}^{(W,S)}(A,\rho,a)$ is the quotient of the free free $k$-algebra generated by  $A$ and by $\{g_s\}_{s\in S}$ by the relations defining $A$ as a $k$-algebra and
\begin{align*} 
g_s c &=(s. c)g_{s},\quad c\in A,\,s\in S\\
\underbrace{g_{s_i}g_{s_j}g_{s_i} \cdots}_\text{$m_{i,j}$ times}&=\underbrace{g_{s_j}g_{s_i}g_{s_j} \cdots}_\text{$m_{i,j}$ times}, \\
g_s^2 &= a_s + (1 - a_s)g_s, \quad s\in S.
\end{align*}
It follows that the datum of a $k$-algebra homomorphism $\mu\colon \mathcal{E}^{(W,S)}(A,\rho,a)\to R$ is equivalently the datum of a $k$-algebra homomorphism $\mu_A\colon A\to R$ together with a choice of elements $r_s$ in $R$ such that
\begin{align*} 
r_s \mu_A(c) &=\mu_A(s. c)r_{s},\qquad c\in A,\\
\underbrace{r_{s_i}r_{s_j}r_{s_i} \cdots}_\text{$m_{i,j}$ times}&=\underbrace{r_{s_j}r_{s_i}r_{s_j} \cdots}_\text{$m_{i,j}$ times}, \\
{r_s}^2 &= \mu_A(a_s) + \mu_A(1 - a_s)r_s.
\end{align*}
\end{remark}

\begin{remark}\label{rem:rescale}
 Let $\{\lambda_s\}_s\in S$ be a collection of invertible elements in $A$ such that $\lambda_{s_i}=\lambda_{s_j}$ whenever $m_{i,j}$ is odd. All of the relations defining  $\mathcal{E}^{(W,S)}(A,\rho,a)$ are homogeneous in the $g_s$'s, with the exception of the relation $g_s^2 = {a_s} + (1 - {a_s})g_s$, the relation $g_s c =(s. c)g_{s}$ is homogeneous in the single variable $g_s$, and the braid-type relations with an even $m_{i,j}$ are separately homogeneous in $g_{s_i}$ and $g_{s_j}$. Therefore, we can operate the rescaling $g_s\rightsquigarrow -\lambda_s g_s$ to obtain a presentation of $\mathcal{E}^{(W,S)}(A,\rho,a)$ with relations
\begin{align*} 
g_s c &=(s. c)g_{s},\\
\underbrace{g_{s_i}g_{s_j}g_{s_i} \cdots}_\text{$m_{i,j}$ times}&=\underbrace{g_{s_j}g_{s_i}g_{s_j} \cdots}_\text{$m_{i,j}$ times}, \\
g_s^2 &= {\lambda_s}^{-2}a_s -  {\lambda_s}^{-1}(1-{a_s})g_s.
\end{align*}
\end{remark}

\begin{remark}\label{rem:commutes}
Since $B_W$ acts on $A$ via its projection on $W$, we have that the elements  $g_s^2$ with $s\in S$ commute with the elements of $A$ in $A \rtimes k[B_W]$ and so in $\mathcal{E}^{(W,S)}(A,\rho,a)$. Indeed, for any $c\in A$,
\begin{align*}
g_s^2c&=(g_s^2cg_s^{-2})g_s^2=(g_s(g_s cg_s^{-1})g_s^{-1})g_s^2\\
&=(g_s(s. c)g_s^{-1})g_s^2=(s.(s. c))g_s^2=(s^2. c)g_s^2=cg_s^{2}.
\end{align*}
As an immediate consequence we have
\[g_scg_s^{-1}=g_s^{-1}g_s^2cg_s^{-1}=g_s^{-1}cg_s^2g_s^{-1}=g_s^{-1}cg_s,\]
for any $s\in \Phi$ and any $c$ in $A$.
\end{remark}
\begin{example}[The Iwahori-Hecke algebra]
Let $k=\mathbb{Z}$. Consider the ring of Laurent polynomials  $\mathbb{Z}[x_{\langle s\rangle} ,x_{\langle s\rangle}^{-1}]$ in variables $x_{\langle s\rangle}$ indexed by the root subsystems ${\langle s\rangle}$ with $s\in S$. The $W$-action on root subsystems of $\Phi$ induces an action on  $\mathbb{Z}[x_{\langle s\rangle} ,x_{\langle s\rangle}^{-1}]$. Let $A_{IH}(W,S)$ be the ring of coinvariants. By definition of coinvariants, the $W$-action on $A_{IH}(W,S)$ is trivial. The map 
\begin{align*}
\hat{u}\colon S &\mapsto A_{IH}(W,S)^\times\\
s & \mapsto [x_{\langle s\rangle}],
\end{align*}
where $ [x_{\langle s\rangle}]$ is the image of $x_{\langle s\rangle}$ in $A_{IH}(W,S)$,
 is $W$-equivariant: if $s_2=w s_1w^{-1}$ with $w\in W$, then we have $\langle s_2\rangle = w. \langle s_1\rangle$ and so
\[
\hat{u}_{s_2}=[x_{\langle s_2\rangle}]=[x_{w. \langle s_1\rangle}]=[w. x_{\langle s_1\rangle}]=[x_{\langle s_1\rangle}]=\hat{u}_{s_1}.
\]
The Marin algebra $
\mathcal{E}^{(W,S)}(A_{IH}(W,S), \rho_{\mathrm{triv}},\hat{u})
$
is then the  Iwahori--Hecke algebra $H{(W,S)}(\hat{u}_s)$. If $R$ is any ring with trivial $W$-action and $u\colon S\to R^\times$ is a $W$-equivariant map, then we have a specialization map $\epsilon\colon A_{IH}(W,S)\to R$ mapping $\hat{u}_s$ to $u_s$; the Marin algebra  is $\mathcal{E}^{(W,S)}(R,\rho_{\mathrm{triv}},u)$ is the Iwahori--Hecke algebra $H_R{(W,S)}(u_s)$ and the specialization map induces a morphism of algebras $H{(W,S)}(\hat{u}_s)\to H_R{(W,S)}(u_s)$.
\end{example}
The specialization map $\epsilon\colon A_{IH}(W,S)\to R$ in the previous example is clearly $W$-equivariant and satisfies $\epsilon\circ \hat{u}=u$. More generally we have the following.
\begin{lemma}\label{lemma:epsilon}
Let $B$ a $k$-algebra with a $W$-action $\eta\colon W\to \mathrm{Aut}_k(B)$ and a $W$-equivariant map $b\colon \Phi\to B^\times$, and let $\epsilon\colon A\to B$ be a $W$-equivariant algebra homomorphism such that $b=\epsilon\circ a$. Then one has an induced algebra homomorphism
\[
\epsilon^{\mathcal{E}}\colon \mathcal{E}^{(W,S)}(A,\rho,a)\to \mathcal{E}^{(W,S)}(B,\eta,b)
\]
that is the identity on $k[B_W]$ and coincides with $\epsilon$ on $A$. In particular, if $\epsilon$ is surjective, then also $\epsilon^{\mathcal{E}}$ is surjective.
\end{lemma} 
Even more generally, we have the functoriality of the construction of the Marin algebra with respect to morphism of Marin data, of which Lemma \ref{lemma:epsilon} is the particular case corresponding to a morphism of Marin data with the same Coxeter system and with the identity map from this system to itself. In the general situation we have the following.
\begin{lemma}\label{lemma:epsilon2}
Let $(\phi,\epsilon)\colon ((W,S),\rho,a)\to ((U,T),\eta,b)$ be a morphism of Marin data. Then one has an induced algebra homomorphism
\[
(\phi,\epsilon)^{\mathcal{E}}\colon \mathcal{E}^{(W,S)}(A,\rho,a)\to \mathcal{E}^{(U,T)}(B,\eta,b),
\]
that is the morphism induced by $\phi$ on $k[B_W]$ and coincides with $\epsilon$ on $A$. In particular, if $\phi\colon S\to T$ and $\epsilon$ are surjective, then also $(\phi,\epsilon)^{\mathcal{E}}$ is surjective.
\end{lemma} 

\begin{example}\label{example:lift}
Let $R$ be a $k$-algebra with trivial $W$-action and $u\colon S\to R^\times$ a $W$-equivariant map. Let $(A,\rho,a)$ be data defining a Marin algebra for $(W,S)$. Then a $W$-equivariant homomorphism $\epsilon\colon A\to R$ such that $\epsilon\circ a=u$ induces an algebra homomorphism
$
\epsilon^{\mathcal{E}}\colon \mathcal{E}^{(W,S)}(A,\rho,a)\to H_R{(W,S)}(u_s).
$
If moreover $A$ is an $R$-algebra and $\epsilon$ is an augmentation, i.e., $R\to A\xrightarrow{\epsilon}R$ is the identity of $R$, then $ \mathcal{E}^{(W,S)}(A,\rho,a)$ is an $R$-algebra and $\epsilon^{\mathcal{E}}$ is a surjective morphism of $R$-algebras. In particular, this means that $H_R{(W,S)}(u_s)$ is realized as a quotient of $\mathcal{E}^{(W,S)}(A,\rho,a)$.
\end{example}
The above example suggests the following.
\begin{definition}
A {\it Marin lift} of the  Iwahori--Hecke algebra $H_R{(W,S)}(u_s)$ is a surjective morphism  $\epsilon^{\mathcal{E}}\colon \mathcal{E}^{(W,S)}(A,\rho,a)\to H_R{(W,S)}(u_s)$ as in example \ref{example:lift}.
\end{definition}

\begin{example}\label{example:monoid}
Let $(R,u)$ be as in the definition the  Iwahori--Hecke algebra \break $H_R{(W,S)}(u_s)$, and let $E$ be a commutative monoid endowed with a $W$-action $\rho$ and a $W$-equivariant map
$
e\colon S\to \mathrm{Idempotents}(E).
$
The monoid $R$-algebra $R[E]$ inherits a $W$-action from the $W$-action on $E$. The map
\begin{equation}\label{he}
\hat{e}=1 + (u - 1)e \colon S\to R[E]
\end{equation}
is clearly $W$-equivariant. Moreover it takes values in $R[E]^\times$, since
$$(1 + (u_s\ - 1)e_s)(1+u_s^{-1}(1-u_s)e_s)=1.$$
So we can form the Marin algebra $\mathcal{E}^{(W,S)}(E,\rho,\hat{\epsilon})$. The algebra $R[E]$ has a canonical augmentation $\epsilon\colon R[E]\to R$ sending all the elements of $E$ to $1$. This augmentation is manifestly $W$-equivariant; moreover $\epsilon\circ \hat{e}=u$, so we have the surjective morphism
$
\epsilon^{\mathcal{E}}\colon \mathcal{E}^{(W,S)}(R[E],\rho,\hat{e})\to H_R{(W,S)}(u_s)
$
realizing a lift of the Iwahori--Hecke algebra.
\end{example}

The following fact is well-known.
\begin{lemma}[Matsumoto]
Let $w$ be an element in $W$ and let $w = s_1 \dots s_k$, with $s_i \in S$,  be a reduced expression for $w$. The element $g_{s_1} \dots g_{s_k} \in \mathcal{E}^{(W,S)}(A,\rho,a)$ does not depend on the particular reduced expression for $w$. Therefore one can define the element $g_w$ in $\mathcal{E}^{(W,S)}(A,\rho,a)$ as $g_w = g_{s_1} \dots g_{s_k}$ for some (hence any) reduced expression for $w$.
\end{lemma}
\begin{lemma}\label{lemma:4-equations}
   The set $\{g_w\}_{w\in W}\subseteq \mathcal{E}^{(W,S)}(A,\rho,a)$ spans  $\mathcal{E}^{(W,S)}(A,\rho,a)$ as a left $A$-module.
\end{lemma}
\begin{proof}
Let $M\subseteq \mathcal{E}^{(W,S)}(A,\rho,a)$ be the left $A$-submodule generated by the elements $\{g_w\}_{w\in W}$. Equivalently, $M$ is the $k$-submodule generated by the elements $\{cg_w\}_{c\in A, w\in W}$. We start by showing that $M$ is a 2-sided ideal of $\mathcal{E}^{(W,S)}(A,\rho,a)$. 
Since as a ring $\mathcal{E}^{(W,S)}(A,\rho,a)$ is generated by the elements $g_s$ of $k[B_W]$ and by the elements $x$ of $A$, we have to show that:
\begin{enumerate}
\item[(1)] $x(cg_w)\in M$;
\item[(2)]  $(cg_w)x\in M$;
\item[(3)]  $g_s(cg_w)\in M$;
\item[(4)]  $(cg_w)g_s\in M$;
\end{enumerate}
for any $s\in S$ and any $x\in A$. Relation (1)  is immediate:
$
x(cg_w)=(xc)g_w.
$\newline
For relation (2) we compute
\[
(cg_w)x=(c(g_wxg_{w}^{-1})) g_w=(c(w. x)) g_w,
\]
where we used that $B_W$ acts on $A$ via the projection $B_W\to W$, and that the image of $g_w$ via this projection is precisely $w$.
To prove (3), we distinguish two cases. We can have either $\ell(sw)=\ell(w)+1$, so that $sw$ is a reduced expression for $sw$, or $\ell(sw)=\ell(w)-1$ so that $w=sw'$ with $w'$ minimally written. Notice that in the second case we have $w'=sw$ in $W$. Now, in the first case we compute
\[
g_s(cg_w) = (g_s c g_s^{-1}) g_sg_w=  (g_s c g_s^{-1}) g_{sw}=(s. c)g_{sw},
\] 
while in the second case we compute
\begin{align*}
g_s(cg_w) &= (g_s c g_s^{-1}) g_sg_w = (g_s c g_s^{-1}) g_sg_s g_{w'}= (g_s c g_s^{-1}) g_s^2 g_{w'}\\
&= (g_s c g_s^{-1}) (a_s + (1 - a_s)g_s) g_{w'}\\
&= (s. c)a_sg_{w'}+ (s. c)(1 - a_s)g_w.
\end{align*}
Relation (4) is checked  similarly. 
Now the conclusion is immediate: since $g_1=1$, we have $1\in M$, and so $M=\mathcal{E}^{(W,S)}(A,\rho,a)$.
\end{proof}

For any $s\in S$ we have two $k$-linear operators on $A$: the multiplication by $a_s$ and the automorphism $\rho_s$.

\begin{lemma}
\label{lemma:free_cond}
    If $\mathcal{E}^{(W,S)}(A,\rho,a)$ is free as an $A$-module with basis $\{g_w\}_{w\in W}$, then,  for any $s\in S$ 
    \begin{equation}\label{eq:free-cond}
    (1 - a_s) (1-\rho_s)= 0 \text{ in }\mathrm{End}_k(A).
    \end{equation}
 \end{lemma}
\begin{proof}
Let  $c\in A$ and $s\in S$. By Remark \ref{rem:commutes}, we have $g_s^2c =  c g_s^2$ and so
\[
     (a_s + (1 - a_s)g_s)c   =c(a_s + (1 - a_s)g_s).
\]
Since $A$ is commutative, this reduces to the identity
\[
     (1 - a_s)g_sc   =(1 - a_s)cg_s.
\]
Writing
$
g_sc=(g_scg_s^{-1})g_s=(s. c)g_s,
$
we find 
$
(1 - a_s)\rho_s(c)g_s=c(1 - a_s)g_s.
$
Since $\{g_w\}_{w\in W}$ is a spanning set  of $\mathcal{E}^{(W,S)}(A,\rho,a)$, this gives $(1 - a_s)(\rho_s(c)-c)=0$.
  
\end{proof}

When $A$ is a free $k$-module, condition \eqref{eq:free-cond} is also sufficient in order to have that $\mathcal{E}^{(W,S)}(A,\rho,a)$ is a free $A$-module over the  basis $\{g_w\}_{w\in W}$. The proof of sufficiency is however more involved and we need a few auxiliary arguments, that we verbatim adapt  from \cite{marin}. We start  by fixing a $k$-linear basis $\{c_i\}_{i\in I}$ for $A$,  and by noticing that the equations derived in the proof of Lemma \ref{lemma:4-equations} force  the effect of left and right multiplication by elements $x$ of $A$ and by elements $g_s$ with $s\in S$ on the elements $c_ig_w$ of $\mathcal{E}^{(W,S)}(A,\rho,a)$ as follows:
\[
x\cdot -\colon c_ig_w\mapsto xc_i g_w.
\]
\[
-\cdot x\colon c_ig_w\mapsto c_i(w. x) g_w.
\]
\[
g_s\cdot -\colon c_ig_w \mapsto \begin{cases}
(s. c_i)g_{sw} & \text{if $\ell(sw)=\ell(w)+1$}\\
\\
(s. c_i)a_sg_{sw}+ (s. c_i)(1 - a_s)g_w& \text{if $\ell(sw)=\ell(w)-1$}
\end{cases}
\]
\[
-\cdot g_s\colon c_ig_w \mapsto \begin{cases}
c_i g_{ws} & \text{if $\ell(sw)=\ell(w)+1$}\\
\\
c_i((ws). a_s)g_{ws}+c_i(1-(ws). a_s)g_w& \text{if $\ell(sw)=\ell(w)-1$}
\end{cases}
\]
Let now $\{\gamma_w\}$ be a collection of elements indexed by $w\in W$ and let $V$ be the free left $E$-module with basis $\{\gamma_w\}_{w\in W}$. It has a $k$-linear basis given by the elements  $\{c_i\gamma_w\}_{i\in I, w\in W}$.
We can then define $k$-linear endomorphisms of $V$ corresponding to the  left and right multiplication by elements $x$ and by elements $g_s$ on the elements of the form $c_ig_w$ of $\mathcal{E}^{(W,S)}$. That is we define $k$-linear operators $C^l_{x},C^r_{x}\colon V\to V$ by 
\begin{align*}
C^l_{x}&\colon c_i\gamma_w\mapsto xc_i \gamma_w,\\
C^r_{x}&\colon c_ig_w\mapsto c_i(w. x) \gamma_w,
\end{align*}
and 
$k$-linear operators $G^l_{s},G^r_{s}\colon V\to V$ by
\begin{align*}
G_s^l&\colon c_i\gamma_w \mapsto \begin{cases}
(s. c_i)\gamma_{sw} & \text{if $\ell(sw)=\ell(w)+1$}\\
\\
(s. c_i)a_s\gamma_{sw}+ (s. c_i)(1 - a_s)\gamma_w& \text{if $\ell(sw)=\ell(w)-1$}
\end{cases}
\\
G_s^r&\colon c_i\gamma_w \mapsto \begin{cases}
c_i \gamma_{ws} & \text{if $\ell(sw)=\ell(w)+1$}\\
\\
c_i((ws). a_s)\gamma_{ws}+c_i(1-(ws). a_s)\gamma_w& \text{if $\ell(sw)=\ell(w)-1$}
\end{cases}
\end{align*}
The key result is the following statement

\begin{proposition}\label{lemma:rho}
 If condition \eqref{eq:free-cond} is satisfied, then the map 
\begin{align*}
\mu(c)=C^{l}_{c},\quad  \mu(g_s)=G^l_s
 \end{align*}
  defines a $k$-algebra homomorphism $\mu\colon \mathcal{E}^{(W,S)}(A,\rho,a)\to \mathrm{End}_{k}(V)$.
\end{proposition}
The proof follows closely that of Theorem 3.4 of \cite{marin}. We provide a few details, to emphasize the role of condition \eqref{eq:free-cond}.\\
First note that, by $k$-linearity,  
for any $c\in A$ we have
\begin{align*}
C^l_{x}\colon c\gamma_w\mapsto xc \gamma_w,\quad
C^r_{x}\colon cg_w\mapsto c(w. x) \gamma_w,
\end{align*}
\begin{align*}
G_s^l&\colon c\gamma_w \mapsto \begin{cases}
(s. c)\gamma_{sw} & \text{if $\ell(sw)=\ell(w)+1$}\\
\\
(s. c)a_s\gamma_{sw}+ (s. c)(1 - a_s)\gamma_w& \text{if $\ell(sw)=\ell(w)-1$}
\end{cases}
\\
G_s^r&\colon c\gamma_w \mapsto \begin{cases}
c \gamma_{ws} & \text{if $\ell(sw)=\ell(w)+1$}\\
\\
c((ws). a_s)\gamma_{ws}+c(1-(ws). a_s)\gamma_w& \text{if $\ell(sw)=\ell(w)-1$}
\end{cases}
\end{align*}

\begin{lemma}\label{lemma:l-and-r-commute}
    If condition \eqref{eq:free-cond} is satisfied, then  $G^l_{s_1} G^r_{s_2} = G^r_{s_2} G^l_{s_1}$ for any $s_1,s_2 \in S$.
\end{lemma}
\begin{proof}
We have to check that $G^l_{s_1} G^r_{s_2}(c\gamma_w)=G^r_{s_2} G^l_{s_1}(c\gamma_w)$ for any $c\in A$ and any $w\in W$.
 There are several cases to be examined and  we will proceed case by case.
$\bullet$  If $l(s_1w)=l(w) + 1$, $l(ws_2) = l(w) + 1$ and $l(s_1ws_2) = l(w) + 2$, then we have:
\begin{align*}
            &G^l_{s_1} G^r_{s_2} (c\gamma_w) =  G^l_{s_1} (c\gamma_{ws_2}) = (s_1. c) \gamma_{s_1ws_2}
        \\
            &G^r_{s_2} G^l_{s_1} (c\gamma_w) = G^r_{s_2} ((s_1. c) \gamma_{s_1w}) = (s_1. c) \gamma_{s_1ws_2}.
    \end{align*}
        
        \item If $l(s_1w)=l(w) + 1$, $l(ws_2) = l(w) + 1$ and $l(s_1ws_2) = l(w)$, then by e.g. \cite[Lemma 3.2]{marin} we have  $s_1ws_2=w$, so that $ws_2w^{-1}=s_1$ and we have:
        \begin{align*}
                   G^l_{s_1} G^r_{s_2} (c\gamma_w) &=  G^l_{s_1} (c\gamma_{ws_2}) = (s_1. c)a_{s_1}\gamma_{s_1ws_2}+ (s_1. c)(1 - a_{s_1})\gamma_{ws_2}\\
                   & = (s_1. c)a_{s_1}\gamma_{w}+ (s_1. c)(1 - a_{s_1})\gamma_{s_1w}
                   \\
                               G^r_{s_2} G^l_{s_1} (c\gamma_w) &= G^r_{s_2} ((s_1. c) \gamma_{s_1w}) \\
                               &= 
                               (s_1. c)((s_1ws_2). a_{s_2})\gamma_{s_1ws_2}+(s_1. c)(1-(s_1ws_2). a_{s_2})\gamma_{s_1w} \\
                               &= 
                               (s_1. c)(w. a_{s_2})\gamma_{w}+(s_1. c)(1-w. a_{s_2})\gamma_{s_1w} \\
                                &= 
                               (s_1. c)( a_{w s_2w^{-1}})\gamma_{w}+(s_1. c)(1- a_{w s_2w^{-1}})\gamma_{s_1w} \\
                                 &= (s_1. c)a_{s_1}\gamma_{w}+(s_1. c)(1- a_{s_1})\gamma_{s_1w}.
        \end{align*}
        
        
        
\noindent        $\bullet$ If $l(s_1w)=l(w) - 1$, $l(ws_2) = l(w) - 1$ and $l(s_1ws_2) = l(w)$ we have that in this case $s_1w=ws_2$. This gives $s_1=ws_2w^{-1}$ and so $w. a_{s_2}=a_{s_1}$. Recalling that by condition \eqref{eq:free-cond}, we have $(1-a_s)(s. c)=(1-a_s)c$, and that $s. a_s=a_s$, we get
     
        \begin{align*}
            &G^l_{s_1} G^r_{s_2} (c\gamma_w)=  G^l_{s_1} (c((ws_2). a_{s_2})\gamma_{ws_2}+c(1-(ws_2). a_{s_2})\gamma_w) \\
            &= (s_1. (c((ws_2). a_{s_2}))\gamma_{s_1ws_2}+(s_1. (c(1-(ws_2). a_{s_2})))a_{s_1}\gamma_{s_1w}\\
            &\qquad\qquad\qquad + (s_1. (c(1-(ws_2). a_{s_2})))(1 - a_{s_1})\gamma_w\\
            &= (s_1. c)((s_1ws_2). a_{s_2})\gamma_{s_1ws_2}+(s_1. c)(1-(s_1ws_2). a_{s_2})a_{s_1}\gamma_{s_1w}\\
            &\qquad\qquad\qquad + (s_1. c)(1-(s_1ws_2). a_{s_2})(1 - a_{s_1})\gamma_w\\
            &= (s_1. c)(w. a_{s_2})\gamma_{w}+(s_1. c)(1-w. a_{s_2})a_{s_1}\gamma_{s_1w}\\
            &\qquad\qquad\qquad + (s_1. c)(1-w. a_{s_2})(1 - a_{s_1})\gamma_w\\
            &= (s_1. c)a_{s_1}\gamma_{w}+(s_1. c)(1-a_{s_1})a_{s_1}\gamma_{s_1w}+ (s_1. c)(1-a_{s_1})(1 - a_{s_1})\gamma_w\\
            &= (s_1. c)a_{s_1}\gamma_{w}+c(1-a_{s_1})a_{s_1}\gamma_{s_1w}+ c(1-a_{s_1})^2\gamma_w,\\ \\
&G^r_{s_2} G^l_{s_1} (c\gamma_w)= G^r_{s_2} ((s_1. c)a_{s_1}\gamma_{s_1w}+ (s_1. c)(1 - a_{s_1})\gamma_w) =\\
                                 &= G^r_{s_2} ((s_1. c)a_{s_1}\gamma_{s_1w}+ c(1 - a_{s_1})\gamma_w) \\
                                 &=(s_1. c)a_{s_1}\gamma_{s_1ws_2}+ c(1 - a_{s_1})(((ws_2). a_{s_2})\gamma_{ws_2}+(1-(ws_2). a_{s_2})\gamma_w)\\
                                 &=(s_1. c)a_{s_1}\gamma_{w}+ c(1 - a_{s_1})((w. a_{s_2})\gamma_{ws_2}+(1-w. a_{s_2})\gamma_w)\\
                                 &=(s_1. c)a_{s_1}\gamma_{w}+ c(1 - a_{s_1})(a_{s_1}\gamma_{s_1w}+(1-a_{s_1})\gamma_w)\\
                                 &=(s_1. c)a_{s_1}\gamma_{w}+ c(1 - a_{s_1})a_{s_1}\gamma_{s_1w}+c(1 - a_{s_1})^2\gamma_w.
           \end{align*}                 
  The  cases   $l(s_1w)=l(w) + 1$, $l(ws_2) = l(w) - 1$ and $l(s_1ws_2) = l(w)$,     $l(s_1w)=l(w) - 1$, $l(ws_2) = l(w) + 1$ and $l(s_1ws_2) = l(w)$  ,   $l(s_1w)=l(w) - 1$, $l(ws_2) = l(w) - 1$ and $l(s_1ws_2) = l(w) - 2$      are treated along the same lines.

\end{proof}

\begin{proof}[Proof of Proposition \ref{lemma:rho}]
Let $\mu_A\colon A\to \mathrm{End}_{k}(V)$ the map defined by $\mu_A(c)=C^{l}_{c}$ for any $c\in A$. By Remark \ref{rem:presentation}, showing that $\mu$ defines a $k$-algebra homomorphism is equivalent to showing that $\mu_A$ is a $k$-algebra homomorphism, i.e., since $\mu_A$ is manifestly $k$-linear, that
$
C^l_{x}C^l_{y}=C^l_{xy}, \ x,y\in A,
$
and that the following identities hold:
\begin{align} 
G^l_s C^l_x &=C^l_{s. x}G^l_s,\label{uno}\\
({G^l_s})^2 &= C^l_{a_s} + C^l_{(1 - a_s)}G^l_s,\label{due}\\
\underbrace{G^l_{s_i}G^l_{s_j}G^l_{s_i} \cdots}_\text{$m_{i,j}$ times}&=\underbrace{G^l_{s_j}G^l_{s_i}G^l_{s_j} \cdots}_\text{$m_{i,j}$ times}.\label{tre}
\end{align}
For \eqref{uno} we compute
\begin{align*}
(C^l_{x}C^l_{y})(c\gamma_w)&=C^l_{x}(C^l_{y}(c\gamma_w))=C^l_{x}((yc)\gamma_w))=xyc\gamma_w=C^l_{xy}(c\gamma_w).
\end{align*}
For \eqref{due}, we find
 \begin{align*}
(G^l_{s}C^l_{x})(c\gamma_w)
&=
\begin{cases}
(s. (xc))\gamma_{sw} & \text{if $\ell(sw)=\ell(w)+1$,}\\
\\
(s. (xc))a_s\gamma_{sw}+ (s. (xc))(1 - a_s)\gamma_w& \text{if $\ell(sw)=\ell(w)-1$,}
\end{cases}
\end{align*}
and
 \begin{align*}
(C^l_{s. c}G^l_s)(c\gamma_w)
&=
\begin{cases}
C^l_{s. c}((s. c)\gamma_{sw}) & \text{if $\ell(sw)=\ell(w)+1$}\\
\\
C^l_{s. c}((s. c)a_s\gamma_{sw}+ (s. c)(1 - a_s)\gamma_w)& \text{if $\ell(sw)=\ell(w)-1$}
\end{cases}
\\
&=
\begin{cases}
(s. (xc))\gamma_{sw} & \text{if $\ell(sw)=\ell(w)+1$}\\
\\
(s. (xc)a_s\gamma_{sw}+(s. (xc))(1 - a_s)\gamma_w& \text{if $\ell(sw)=\ell(w)-1$}
\end{cases}
\end{align*}
For \eqref{tre} we compute 
 \begin{align*}
(G^l_{s})^2(c\gamma_w)&=
\begin{cases}
G^l_{s}((s. c)\gamma_{sw}) \quad \text{if $\ell(sw)=\ell(w)+1$}\\
\\
G^l_{s}((s. c)a_s\gamma_{sw}+ (s. c)(1 - a_s)\gamma_w)& \text{if $\ell(sw)=\ell(w)-1$}
\end{cases}\\
&=
\begin{cases}
((s. ((s. c)a_s))\gamma_{w}+ (s. (s. c))(1 - a_s)\gamma_{sw} & \text{if $\ell(sw)=\ell(w)+1$}\\
\\
(s. ((s. c)a_s))\gamma_{w}+ ((s. (s. c))a_s\gamma_{sw}+ (s.(s. c))(1 - a_s)\gamma_w \\
 - ((s. ((s. c)a_s))a_s\gamma_{sw}+ (s. ((s. c)a_s))(1 - a_s)\gamma_w)& \text{if $\ell(sw)=\ell(w)-1$}
\end{cases}
\\
&=
\begin{cases}
a_sc\gamma_{w}+ (1 - a_s)c\gamma_{sw} & \text{if $\ell(sw)=\ell(w)+1$}\\
\\
ca_s\gamma_{w}+ c a_s\gamma_{sw}+ c(1 - a_s)\gamma_w  - c(a_s)^2\gamma_{sw}- ca_s(1 - a_s)\gamma_w & \text{if $\ell(sw)=\ell(w)-1$}
\end{cases}
\\
&=
\begin{cases}
a_sc\gamma_{w}+ (1 - a_s)c\gamma_{sw} & \text{if $\ell(sw)=\ell(w)+1$}\\
\\
 a_s(1-a_s)c\gamma_{sw} + (1-a_s(1 - a_s))c\gamma_w & \text{if $\ell(sw)=\ell(w)-1$}.
\end{cases}
\end{align*}
By condition \eqref{eq:free-cond}, we have $(1-a_s)(s. c)=(1-a_s)c$, so we find
\begin{align*}
&\phantom{xx}(C^l_{a_s} + C^l_{(1 - a_s)}G^l_s)(c\gamma_w)=\\&=C^l_{a_s}(c\gamma_w)+C^l_{(1 - a_s)}(G^l_s(c\gamma_w))\\
&=
\begin{cases}
C^l_{a_s}(c\gamma_w)+C^l_{(1 - a_s)}((s. c)\gamma_{sw})   \text{if $\ell(sw)=\ell(w)+1$}\\
\\
C^l_{a_s}(c\gamma_w)+C^l_{(1 - a_s)}((s. c)a_s\gamma_{sw}+ (s. c)(1 - a_s)\gamma_w)& \text{if $\ell(sw)=\ell(w)-1$}
\end{cases}\\
&=
\begin{cases}
a_sc\gamma_w+(1 - a_s)(s. c)\gamma_{sw} & \text{if $\ell(sw)=\ell(w)+1$}\\
\\
a_sc\gamma_w+(1 - a_s)(s. c)a_s\gamma_{sw}+ (1 - a_s)(s. c)(1 - a_s)\gamma_w& \text{if $\ell(sw)=\ell(w)-1$}
\end{cases}\\
&=
\begin{cases}
a_sc\gamma_w+(1 - a_s)c\gamma_{sw} & \text{if $\ell(sw)=\ell(w)+1$}\\
\\
a_s(1 - a_s)c\gamma_{sw}+ (1 - a_s(1-a_s))c\gamma_w& \text{if $\ell(sw)=\ell(w)-1$}
\end{cases}
\end{align*}
Finally, we need to show that the $G^l_s$ satisfy the braid relations. Here we can proceed exactly  as in \cite{marin}.
\end{proof}
\begin{remark}\label{rem:left-is-left}
Notice that the left $A$-action on $V$ induced by $\mu$ via the inclusion $A\hookrightarrow \mathcal{E}^{(W,S)}(A,\rho,a)$ coincides with the left $A$-action on $V$ given by its $A$-module structure. 
\end{remark}

\begin{corollary}\label{cor:free}
Assume $A$ is a free $k$-module. Then $\mathcal{E}^{(W,S)}(A,\rho,a)$ is the free $A$-module over the  basis $\{g_w\}_{w\in W}$ if and only if
 condition \eqref{eq:free-cond} is satisfied.
\end{corollary}
\begin{proof}
The ``only if'' part is Lemma \ref{lemma:free_cond}. To prove the ``if'' part, 
consider the free $A$-module $V$ on the basis $\{\gamma_{w}\}_{w\in W}$ endowed with the left $\mathcal{E}^{(W,S)}(A,\rho,a)$-module structure on the $k$-module $V$ defined by the algebra homomorphism $$\mu\colon \mathcal{E}^{(W,S)}(A,\rho,a)\to \mathrm{End}_k(V)$$ from Lemma \ref{lemma:rho}. Then we have a $k$-linear map
$
\lambda \colon \mathcal{E}^{(W,S)}(A,\rho,a) \to V,\,\lambda(x)= x._\mu \gamma_{1}.
$
By Remark \ref{rem:left-is-left}, the map $\lambda$ is a map of left $A$-modules. Since $V$ is free as an $A$-module, we can define an $A$-module map $\eta: V\to \mathcal{E}^{(W,S)}(A,\rho,a)$ by defining it on the $A$-basis $\{\gamma_w\}_{w\in W}$. We set $\eta(\gamma_w)=g_w.$ 
The morphisms of $A$-modules $\lambda$ and $\eta$ are inverse each other. To see this, it suffices to check $\lambda\circ\eta=\mathrm{id}_V$ and $\eta\circ\lambda=\mathrm{id}_{\mathcal{E}^{(W,S)}(A,\rho,a)}$ on $A$-generating sets for $V$ and $\mathcal{E}^{(W,S)}(A,\rho,a)$. For $V$ we can take as generating set the $A$-basis $\{\gamma_w\}_{w\in W}$. We find
\[
(\lambda\circ\eta)(\gamma_w)=\lambda(g_w)=g_w._\mu \gamma_{1}.
\]
If $w=s_1\cdots s_k$ a reduced expression for $w$, then $g_w=g_{s_1}\cdots g_{s_k}$ and so
\begin{align}\label{eq:gwgamma0}
g_w._\mu \gamma_{1}&=g_{s_1}._\mu(g_{s_1}._\mu\cdots ._\mu(g_{s_k}._\mu g_1)\cdots)\\
\notag &=G^l_{s_1}\cdots G^l_{s_k}(\gamma_1)=G^l_{s_1}\cdots G^l_{s_{k-1}}(\gamma_{s_k})=\dots\\
\notag &=\gamma_{s_1\cdots s_k}=\gamma_w.
\end{align}
Therefore
$
(\lambda\circ\eta)(\gamma_w)=\gamma_w, 
$
for any $w\in W$.
By Lemma \ref{lemma:4-equations}, the set $\{g_w\}){w\in W}$ is a generating set for $\mathcal{E}^{(W,S)}(A,\rho,a)$ as an $A$-module. We find, by using \eqref{eq:gwgamma0} again,
\[
(\eta\circ\lambda)(g_w)=\eta(\lambda(g_w))=\eta(g_w._\mu \gamma_1)=\eta(\gamma_w)=g_w,
\]
for any $w\in W$.
\end{proof}

\begin{corollary}
If $A$ is free of finite rank over $k$ and condition \eqref{eq:free-cond} is satisfied, then $\mathcal{E}^{(W,S)}(A,\rho,a)$ is free of finite rank over $k$ and
\begin{equation}\label{rk}
\mathrm{rk}_k \mathcal{E}^{(W,S)}(A,\rho,a)= |W|\, \mathrm{rk}_{k}A.
\end{equation}
A $k$-linear basis of $\mathcal{E}^{(W,S)}(A,\rho,a)$ is $\{c_i g_w\}_{w\in W}$ and $\{c_i\}_{i\in I}$ a $k$-basis of $A$. 
\end{corollary}


\begin{remark}
In the same assumptions as in Lemma \ref{lemma:epsilon}, we have
\[
\epsilon\circ((1 - a_s) (1-\rho_s))=((1 - b_s) (1-\eta_s))\circ \epsilon\colon A \to B
\]
In particular, if $\epsilon$ is surjective and $(1 - a_s) (1-\rho_s)=0$ then also $(1 - b_s) (1-\eta_s)$, so that if $\epsilon$ is surjective and $\mathcal{E}^{(W,S)}(A,\rho,a)$ is free as an  $A$-module, then $\mathcal{E}^{(W,S)}(B,\eta,b)$ is free as a $B$-module.
\end{remark}
\begin{remark}
One immediately sees form Corollary \ref{cor:free} that if the $W$-action $\rho$ on $A$ is trivial, then $\mathcal{E}^{(W,S)}(A,a)$ is a free $A$-module with basis $\{g_w | w\in W\}$. In particular, one recovers this way the well known result that  $H_R{(W,S)}(u_s)$ is a free $R$-module with basis $\{g_w | w\in W\}$.
\end{remark}
\begin{remark}\label{rem:mih}
In the same notation as in Example \ref{example:monoid}, one has
\[
(1 - \hat{\epsilon}_s) (1-\rho_s)= (1-u_s)e_s (1-\rho_s).  
\]
In particular, if 
\begin{equation}\label{eq:mih}
e_s (1-\rho_s)=0 \text{ for all }s\in S
\end{equation}
one has that $\mathcal{E}^{(W,S)}(R[E],\rho,\hat{e})$ is a free $R[E]$-module with basis $\{g_w | w\in W\}$.
\end{remark}
Example \ref{example:monoid} and Remark \ref{rem:mih} motivate the following
\begin{definition}
A {\it Marin-Iwahori-Hecke monoid} is a triple $(E,\rho,e)$ consisting of a commutative monoid $E$, a $W$-action $\rho$ on the monoid $E$, and a $W$-equivariant map $
e\colon S\to \mathrm{Idempotents}(E)$ satisfying equation \eqref{eq:mih}. 
\end{definition}
\begin{example}[Marin's algebras]\label{ex:marin-roots}
Let $\Sigma(\Phi)$ be the monoid of root subsystems of $\Phi$, with the composition law is given by $\Psi_1\cdot \Psi_2=\langle \Psi_1,\Psi_2\rangle$, where $\langle X\rangle$ denotes the root subsystem generated by $X$. The $W$-action  on $\Phi$ induces a $W$-action on $\Sigma(\Phi)$. Let us denote this action by $\rho$. All elements of $\Sigma(\Phi)$ are idempotent, and the map 
\begin{equation}\label{e}
e\colon S \to \Sigma(\Phi),\quad 
e(s)=\langle s\rangle
\end{equation}
is manifestly $W$-equivariant. Let $R$ be a commutative $k$-algebra with trivial $W$-action, and with a $W$-equivariant map $u\colon S\to R^\times$. Since $R[\Sigma(\Phi)]$ is generated by the root subsystems $\langle \alpha\rangle$ with $\alpha\in \Phi$, equation \eqref{eq:mih} is equivalent to
\[
(e_s (1-\rho_s))(\langle \alpha\rangle)=0 \text{ for all }s\in S \text{ and all }\alpha\in \Phi.
\]
This is the equation $
\langle s\rangle\cdot \langle \alpha\rangle = \langle s\rangle\cdot \langle s. \alpha\rangle,
$
where $.$ is the $W$-action on $\Phi$, i.e., the equation
$
\langle s,\alpha\rangle = \langle s, s. \alpha\rangle.
$
That this relation is satisfied is a general fact in the theory of Coxeter groups. Hence $(\Sigma(\Phi),\rho,e)$ is a Marin-Iwahori-Hecke monoid. It follows that $\mathcal{E}^{(W,S)}(R[\Sigma(\Phi)],\rho,\hat{e})$ is a free $R[\Sigma(\Phi)]$-module with basis $\{g_w | w\in W\}$. Moreover, the augmentation map $\epsilon\colon R[\Sigma(\Phi)]\to R$ induces a surjective algebra homomorphism \begin{equation}\label{homo}
\mathcal{E}^{(W,S)}(R[\Sigma(\Phi)],\rho,\hat{e})\to H_R{(W,S)}(u_s).
\end{equation}
\end{example}
\noindent The algebra $\mathcal{E}^{(W,S)}(R[\Sigma(\Phi)],\rho,\hat{e})$ is precisely denoted by $\mathcal{C}_W$ in \cite{marin} mentioned in the Introduction.
\begin{lemma}\label{lemma:generators}
Let $\{x_i\}$ be a set of generators for the commutative monoid $E$. Then condition $\eqref{eq:mih}$ is equivalent to
\begin{equation}\label{eq:mih-generators}
e_s\cdot x_i=e_s\cdot (s. x_i)\quad\text{ for any } s\in S \text{ and any generator }x_i. 
\end{equation}
\end{lemma}
\begin{proof}
Equation $\eqref{eq:mih}$ is equivalent to 
\begin{equation}\label{eq:fot-internal-reference}
e_s\cdot x=e_s\cdot (s. x)
\end{equation}
for any $s\in S$ and any $x\in E$. So we only need to show that if $x$ and $y$ satisfy equation \eqref{eq:fot-internal-reference}, then also $x\cdot y$ satisfies it. We have
\begin{align*}
e_s\cdot (s. (x\cdot y))&=e_s\cdot ((s. x)\cdot(s.  y))=(e_s\cdot (s. x))\cdot(s.  y)=(e_s\cdot  x)\cdot(s.  y)\\
&=(x\cdot e_s)\cdot(s.  y)=x\cdot (e_s\cdot(s.  y))=x\cdot (e_s\cdot y)\\
&=e_s\cdot(x\cdot y).
\end{align*}
Notice how the commutativity of $E$ played an essential role in the proof.
\end{proof}
\begin{remark}
Condition \eqref{eq:mih}, which amounts to  $e_s\cdot x=e_s\cdot (s. x)$ for any $s\in S$ and any $x\in E$, can be quite restrictive. One can relax it by looking at the $W$-invariant submonoid $E_{[e]}$ generated by the image of $e\colon S\to \mathrm{Idempotents}(E)$, i.e., the submonoid generated by the elements $w. e_s$ with $w\in W$ and $s\in S$.  The $W$-action on $E$ restricts to a $W$-action on $E_{[e]}$, and $e$ itself can be seen as a $W$-equivariant map $e\colon S\to \mathrm{Idempotents}(E_{[e]})$. Then, by Lemma \ref{lemma:generators}, one sees that $(E_{[e]},\rho,e)$ is a Marin-Iwahori-Hecke monoid precisely when  
\begin{equation}\label{eq:mih-generators-e}
e_{s_1}\cdot (w. e_{s_2})=e_{s_1}\cdot ((s_1w). e_{s_2})\quad\text{ for any } s_1,s_2\in S, \, w\in W. 
\end{equation}
\end{remark}\vskip5pt
The above remark suggests an immediate generalization of Example \ref{ex:marin-roots}, where on considers two Coxeter systems instead of a single one. We begin with the following.
\begin{definition}\label{def:j-map}
Let $(W,S)$ and $(U,T)$ be Coxeter systems.
    A pair $(\phi,e)$ where $\phi : W \rightarrow U$ is an homomorphisms of groups (not necessarily a morphism of Coxeter systems) and $e$ is a map of sets 
  $e\colon  S \rightarrow \Sigma(\Phi_U),
    $
    is called a {\it Juyumaya map} if  $e$ is $W$-equivariant  (via $\phi$) and 
\begin{equation}\label{JM}\langle e_{s_1},w. e_{s_2}\rangle  = \langle e_{s_1},(s_1w). e_{ s_2}\rangle \text{ for any $s_1,s_2\in S, w\in W$},\end{equation}
    where $\langle X\rangle$ denotes the root subsystem generated by $X$ and $.$ denotes the $W$-action on $\Phi_W$.
\end{definition}

\begin{example}\label{ex:juyumaya}
Let $\phi\colon (W,S)\to (U,T)$ be a morphism of Coxeter systems, and let 
\begin{align*}
e_\phi\colon S \to \Sigma(\Phi_U),\quad 
e_\phi(s)=\langle \phi(s)\rangle.
\end{align*}
 The pair $(\phi,e_\phi)$, where the $\phi$  is the group homomorphism $\phi\colon W\to U$ associated with the morphism of Coxeter systems, is a Juyumaya map. Indeed, since a morphism of Coxeter systems preserves the Coxeter matrix, the map $\phi\colon S\to T$ is $W$-equivariant, and so also $e_\phi$ is $W$-equivariant. Next, as in Example \ref{ex:marin-roots}, the condition $\langle e_{\phi;s_1},w. e_{\phi;s_2}\rangle  = \langle e_{\phi;s_1}, (s_1w). e_{\phi;s_2}\rangle $ is equivalent to 
\[
\langle \phi(s_1),\phi(w). \phi(s_2)\rangle = \langle \phi(s_1),\phi(s_1). (\phi(w).\phi(s_2))\rangle.
\]
Here 
we used the $W$-equivariance of $\phi$ and the fact that $W$ acts on $\Phi_U$, and so on $\Sigma(\Phi_U)$, via $\phi\colon W\to U$. Since $\phi(s_1)\in T$ and $w. \phi(s_2)\in \Phi_U$, the conclusion follows as in Example \ref{ex:marin-roots}.
\end{example}

Let $(\phi,e)$ be a Juyumaya map. 
Recalling the definition of the product in $\Sigma(\Phi_U)$, we  see that the second condition in the definition of a  Juyumaya  map is equivalent to
\[
e_{s_1}\cdot (w. e_{s_2})  =e_{s_1} \cdot ((s_1w). e_{s_2})\quad \text{ for any $s_1,s_2\in S$}.
\]
But this is precisely condition \eqref{eq:mih-generators-e}. In other words we have proven the following.
\begin{theorem}\label{prop:juyumaya}
Let $(W,S)$ and $(U,T)$ be Coxeter systems and let $(\phi,e)$ be a Juyumaya map between them. Let $(R,u)$ be a pair consisting of a commutative  ring $R$ with trivial $W$-action and  a $W$-equivariant map $u\colon S\to R^\times$. Let $\rho_\phi$ denote  the $W$-action on $\Sigma(\Phi_U)$ induced by $\phi\colon W\to U$ and by the canonical action of $U$ on $\Phi_U$. Then $(\Sigma(\Phi_U)_{[e]},\rho_\phi,e)$ is a Marin-Iwahori-Hecke monoid. It follows that $\mathcal{E}^{(W,S)}(R[\Sigma(\Phi_U)_{[e]}],\rho_\phi,\hat{e})$ is a free $R[\Sigma(\Phi_U)_{[e]}]$-module with basis $\{g_w | w\in W\}$ ($\hat{e}$  is as in \eqref{he}). Moreover, the augmentation map $\epsilon\colon R[\Sigma(\Phi_U)_{[e]}]\to R$ induces a surjective algebra homomorphism $
\mathcal{E}^{(W,S)}(R[\Sigma(\Phi_U)_{[e]}],\rho_\phi,\hat{e})\to H_R{(W,S)}(u_s).
$
\end{theorem}
The following Lemmas  will be instrumental to provide nontrivial examples of Juyumaya maps, which will be used in applying Proposition \ref{prop:juyumaya}.

\begin{lemma}\label{lemma:representatives-are-enough}
Let $(W,S)$ and $(U,T)$ be Coxeter systems, and let $\phi\colon W\to U$ be a group homomorphism (not necessarily a morphism of Coxeter systems). Let $e\colon S\to \Sigma(\Phi_U)$ be a $W$-equivariant map (via $\phi$), and let $\mathcal{S}\subseteq S$ be a set of representatives for the equivalence relation on $S$ induced by the  conjugation action of $W$ on itself. Then the following are equivalent.
\begin{enumerate}
\item $(\phi,e)$ is a Juyumaya map;
\item $\langle e_{s},w. e_{s_2}\rangle  = \langle e_{s},(sw). e_{s_2}\rangle $ for each $s\in \mathcal{S}$, $s_2\in S$ and $w\in W$.
\end{enumerate}
\end{lemma}
\begin{proof}
One implication is clear. To prove the other, let $s_1$ be an element in $S$. Then there exist $s$ in $\mathcal{S}$ and $w_1$ in $W$ such that $s_1=w_1sw_1^{-1}$. Since $e$ is $W$-equivariant, we have
\begin{align*}
\langle e_{s_1},w. e_{s_2}\rangle&=\langle e_{w_1sw_1^{-1}},w. e_{s_2}\rangle=\langle w. e_{s},w_1. ((w_1^{-1}w). e_{s_2})\rangle\\
&=w_1.\langle e_{s}, (w_1^{-1}w). e_{s_2}\rangle=w_1.\langle e_{s},  (sw_1^{-1}w). e_{s_2}\rangle\\
&=\langle w_1. e_{s}, (w_1sw_1^{-1}w). e_{s_2}\rangle=\langle e_{w_1sw_1^{-1}}, (w_1sw_1^{-1}w). e_{s_2}\rangle\\
&= \langle e_{s_1}, (s_1w). e_{s_2}\rangle.
\end{align*}
\end{proof}

\begin{definition}
Let $(W,S)$ and $(U,T)$ be Coxeter systems, and let $\phi\colon W\to U$ be a group homomorphism (not necessarily a morphism of Coxeter systems). 
Let $\mathcal{S}\subseteq S$ be a set of representatives for the equivalence relation on $S$ induced by the conjugation action of $W$ on itself.
For every $s\in \mathcal{S}$, let $K_s\subseteq T$ be the set
\[
K_s=\{t \in T\text{ such that } \phi(C_W(s))\subseteq C_U(t)\}.
\]
 We say that the pair $(\phi,\mathcal{S})$ is a {\it Juyumaya pair} if $K_s$ is nonepty for every $s\in \mathcal{S}$.
\end{definition}
\begin{example}
If  $\phi\colon (W,S)\to (U,T)$ is a morphism of Coxeter systems, then $(\phi,\mathcal{S})$ is a Juyumaya pair for any set of representatives $\mathcal{S}$. Indeed, the element $\phi(s)$ is an element of $T$ in this case, and $\phi(s)\in K_s$.
\end{example}

\begin{lemma}\label{lemma:equivariant}
Let $(W,S)$ and $(U,T)$ be Coxeter systems, and let $(\phi,\mathcal{S})$ be a  Juyumaya pair. Let
$
   t\colon \mathcal{S}\to T
$
   be a map such that $t(s)\in K_s$ for every $s\in \mathcal{S}$. Then there exists a unique $W$-equivariant map $e_t \colon S \rightarrow \Sigma(\Phi_U)$ such that $e_{t;s}=\langle t(s) \rangle$ for every $s\in \mathcal{S}$. 
\end{lemma}
\begin{proof}
Let $s\in S$. Then there exist $w\in W$ and $s_0\in \mathcal{S}$ such that $s=ws_0w^{-1}$. We set
$
e_t\colon s\mapsto \phi(w). \langle t(s_0) \rangle.
$
We need to show this is a well-posed definition. If $s=\tilde{w} \tilde{s}_0\tilde{w}^{-1}$ for another pair $(\tilde{w},\tilde{s}_0)$, then $s_0$ and $\tilde{s}_0$ are in the same $W$-orbit. Since both $s$ and $\tilde{s}$ are in $\mathcal{S}$ this implies $\tilde{s}_0=s_0$. Then $w^{-1}\tilde{w}$ centralizes $s_0$ and so $\phi(w^{-1}\tilde{w})$ centralizes $t(s_0)$. Therefore $\phi(\tilde{w}). \langle t (\tilde{s}_0)\rangle=\phi(\tilde{w}). \langle t(s_0) \rangle=\phi(w). \langle t(s_0) \rangle$. The map $e$ is $W$-equivariant: if $(s_1,s_2,w)\in S\times S\times W$ is a triple such that $ws_1w^{-1}=s_2$, let $s_0\in \mathcal{S}$ and $w_0\in W$ be such that $s_1=w_0s_0w_0^{-1}$. Then  $s_2=(ww_0)s_0(ww_0)^{-1}$ and so
\begin{align*}
e_{t;s_2}
&=(\phi(ww_0)). \langle t(s_0) \rangle=(\phi(w)\phi(w_0)). \langle t(s_0) \rangle\\
&=\phi(w).(\phi(w_0). \langle t(s_0) \rangle)=\phi(w). e_{t;s_1}.
\end{align*}
Uniqueness is clear: if $\tilde{e}$ is another $W$-equivariant map with $\tilde{e}\vert_{\mathcal{S}}=\langle t \rangle$, then for any $s\in S$ we have
$
\tilde{e}_{s}=\tilde{e}_{ws_0w^{-1}}=\phi(w).\tilde{e}_{s_0}=\phi(w). \langle t(s_0) \rangle=e_{t;s},
$
where $(w,s)$ is any pair in $W\times \mathcal{S}$ such that $s=w s_0w^{-1}$. 
\end{proof}
\begin{definition}
In the same notation as in Lemma \ref{lemma:equivariant}, we say that $(\phi,\mathcal{S},t)$ is a {\it  Juyumaya triple} if $\phi(s)\in\{1,t(s)\}$ for any $s\in \mathcal{S}$.
\end{definition}
\begin{lemma}
Let $(\phi,\mathcal{S},t)$ be a Juyumaya triple. Then $(\phi,e_t)$ is a Juyumaya map, where $e_t\colon S \rightarrow \Sigma(\Phi_U)$ is the map defined in Lemma \ref{lemma:equivariant}.
\end{lemma}
\begin{proof}
By Lemma \ref{lemma:representatives-are-enough} we only need to check that
$\langle e_{t;s},w. e_{t;s_2}\rangle  = \langle e_{t;s},(sw). e_{t;s_2}\rangle $ for each $s\in \mathcal{S}$, $s_2\in S$  and $w\in W$. This is equivalent to 
\begin{equation}\label{eq:two-cases}
\langle e_{t;s},\phi(w). e_{t;s_2}\rangle  = \langle e_{t;s},\phi(s). (\phi(w). e_{t;s_2})\rangle
\end{equation}
 for each $s\in \mathcal{S}$, $s_2\in S$  and $w\in W$. By assumption we have two possibilities for $s\in \mathcal{S}$: either $\phi(s)=1$ or $\phi(s)=t(s)$. In the first case, equation \eqref{eq:two-cases} is trivially satisfied. In the second case \eqref{eq:two-cases} becomes 
\[
\langle t(s),\phi(w). e_{t;s_2}\rangle  = \langle t(s),t(s). (\phi(w). e_{t;s_2})\rangle.
\]
We have $t(s)=t_0$ for some $t_0\in T$ and $\phi(w). e_{t;s_2}=\langle \beta \rangle$ for some $\beta$ in $\Phi_U$. We then use again the general fact from the theory of Coxeter groups that for any $t_0$ in $T$ and any $\beta$ in $\Phi_U$ one has $\langle t_0, \beta \rangle=\langle t_0,t_0. \beta \rangle$, to deduce
$\langle t_0,\langle \beta \rangle \rangle=\langle t_0,t_0. \langle \beta \rangle\rangle$.  
\end{proof}
\begin{lemma}\label{lemma:to-produce-triples}
Let $(W,S)$ and $(U,T)$ be Coxeter systems, and let $\phi\colon W\to U$ be a group homomorphism (not necessarily a morphism of Coxeter systems) such that $\phi(S)\subseteq T\cup \{1\}$. Let $\mathcal{S}\subseteq S$ be a set of representatives for the equivalence relation on $S$ induced by the conjugation action of $W$ on itself, and let $\mathcal{S}_0=\mathcal{S}\cap \ker(\phi)$. Assume $K_s\neq \emptyset$ for any $s\in \mathcal{S}_0$, and let 
$
   t_0\colon \mathcal{S}_0\to T
$
   be a map such that $t(s)\in K_s$ for every $s\in \mathcal{S}_0$. Let
$
   t\colon \mathcal{S}\to T
$
   be the map defined by
   \[
   t(s)=\begin{cases}
   t_0(s)&\text{ if } s\in \mathcal{S}_0,\\
   \phi(s) &\text{ if } s\not \in \mathcal{S}_0.
   \end{cases}
   \]  
 Then $(\phi,\mathcal{S},t)$ is a Juyumaya triple, and so $(\phi, e_t)$ is a Juyumaya map, where $e_t$ is the map from Lemma \ref{lemma:equivariant}.
\end{lemma}
\begin{proof}
Clearly $\phi(s)\in\{1,t(s)\}$ for any $s\in\mathcal{S}$ so we only need to show that $(\phi,\mathcal{S})$ is a Juyumaya pair and that $t(s)\in K_s$ for any $s\in \mathcal{S}$. If $s\in \mathcal{S}_0$, then $K_s\neq \emptyset$ and $t(s)=t_0(s)\in K_s$ by assumption; if $s\not\in \mathcal{S}_0$, then $\phi(s)\in T$ and $t(s)=\phi(s)\in K_s$.
\end{proof}
\begin{remark}\label{rem:morpism-of-coxeter-systems}
If $\phi\colon (W,S)\to (U,T)$ is a morphism of Coxeter systems, then $\phi(S)\subseteq T$ and the assumptions of Lemma \ref{lemma:to-produce-triples} 
are trivially satisfied for any $\mathcal{S}$, since $\mathcal{S}_0=\emptyset$ in this case. It follows that the map $t$ is the restriction of $\phi$ to $\mathcal{S}$. By uniqueness, $e_t=e_\phi$, where $e_\phi$ is the map defined in  Example \ref{ex:juyumaya}. This way we recover Example \ref{ex:juyumaya} as a particular case of Lemma \ref{lemma:to-produce-triples}.
\end{remark}
\begin{lemma}
Let $(W,S)$, $(U,T)$ and $(Z,R)$ be Coxeter systems, and let $(\phi,\mathcal{S})$, with $\phi\colon W\to U$ and $\mathcal{S}\subseteq S$, be a pair satisfying the assumptions of Lemma \ref{lemma:to-produce-triples}, and $\psi\colon (U,T)\to (Z,R)$ be a morphism of Coxeter systems. Then the pair $(\psi\circ\phi,\mathcal{S})$ satisfies the assumptions of Lemma \ref{lemma:to-produce-triples}.
\end{lemma}
\begin{proof}
For any $s\in S$ we have $(\psi\circ\phi)(s)=\psi(1)$ or $(\psi\circ\phi)(s)=\psi(t)$ for some $t\in T$. In either case we have that $(\psi\circ\phi)(s)\in R\cup\{1\}$, since $\psi$ is a morphism of Coxeter systems. Let $s\in \mathcal{S}_0$. Then there exists $t\in T$ such that $\phi(C_W(s))\subseteq C_U(t)$. Let $u\in U$ be the element $u=\psi(t)$. Then we have 
$
(\psi\circ\phi)(C_W(s))\subseteq \psi(C_U(t))\subseteq C_Z(u).
$
\end{proof}
\begin{lemma}\label{lem:composition}
Let $(W,S)$, $(U,T)$ and $(Z,R)$ be Coxeter systems, and let $(\phi,\mathcal{S})$, with $\phi\colon W\to U$ and $\mathcal{S}\subseteq S$, be a pair satisfying the assumptions of Lemma \ref{lemma:to-produce-triples}, and $\psi\colon (Z,R)\to (W,S)$ be a morphism of Coxeter systems. Let $\mathcal{R}\subseteq R$ be a set of representatives for the equivalence relation on $R$ induced by the conjugation action of $Z$ on itself. If $\phi(\mathcal{R})\subseteq \mathcal{S}$, then the pair $(\phi\circ\psi,\mathcal{R})$ satisfies the assumptions of Lemma \ref{lemma:to-produce-triples}.
\end{lemma}
\begin{proof}
Since the image of $\phi\circ\psi$ is contained in the image of $\phi$ we have $(\phi\circ\psi)(R)\subseteq T\cup\{1\}$. Let $\mathcal{R}_0=\mathcal{R}\cap \ker(\phi\circ\psi)$. Since $\psi$ is a morphism of Coxeter systems, if $r\in \mathcal{R}_0$ then $\psi(r)\in \mathcal{S}\cap \ker(\phi)=\mathcal{S}_0$. Then there exists $t\in T$ such that $\phi(C_{W}(\psi(r))\subseteq C_U(t)$, and so $(\phi\circ\psi)(C_Z(r))=\phi(\psi(C_Z(r)))\subseteq \phi(C_W(\psi(r))\subseteq C_U(t)$.
\end{proof}

\section{Examples}
We here provide examples of pairs $(\phi,\mathcal{S})$ satisfying the assumptions of Lemma \ref{lemma:to-produce-triples}, and so producing Juyumaya maps. 


We denote a root system of finite type $X$ (resp. affine type $\widehat X$) with Coxeter diagram of vertices $i_1,\ldots, i_n$ by $X_{i_1,\ldots, i_n}$ (resp. $\widehat X_{i_1,\ldots, i_n}$); the corresponding reflection group is denoted by $W( X_{i_1,\ldots, i_n})$. When no confusion is possible, we will denote a morphism of Coxeter systems 
\[
\phi\colon (W(X_{i_1,\ldots, i_n});\{s_{i_1},\dots,s_{i_n}\})\to (W(Y_{j_1,\ldots, j_m});\{t_{j_1},\dots,t_{j_m}\})
\] 
simply by $\phi\colon X\to Y$. A similar convention applies when one or both of the Coxeter diagrams are of affine type.

The first series of examples derives from Remark \ref{rem:morpism-of-coxeter-systems}, so they are examples of morphisms of Coxeter systems.

\begin{example}\label{ex:terminal}
The Coxeter system $A_1$ is terminal among Coxeter systems: for any Coxeter system $X_{i_1,\dots, i_n}$ there exists a unique morphism of Coxeter systems 
\[
\phi\colon (W(X_{i_1,\ldots, i_n});\{s_{i_1},\dots,s_{i_n}\})\to (W(A_1);\{t_{1}\}).
\] 
Namely, the unique map $\{s_{i_1},\dots,s_{i_n}\}\to \{t_1\}$ extends to a morphisms of groups $W(X_{i_1,\ldots, i_n})\to W(A_1)\cong S_2$.
\end{example}

\begin{example}\label{ex:A-to-A}
For any Coxeter system $X$, the identity of $X$ is a morphism of Coxeter systems. In particular, if we consider the Coxeter diagrams
\[
A_{1\ldots n}:  \dynkin[labels={1,2,3,,n-1,n},             scale=2] A{ooo...ooo}
\]
and
\[
A_{1\ldots n}:  \dynkin[labels={1,2,3,,n-1,n},             scale=2] A{ooo...ooo}
\]
then the map 
\[
\phi(s_i)=t_i
\] 
defines a morphism of Coxeter systems $\varphi\colon A\to A$. Despite being trivial, this example will play a role in the following section. 
\end{example}

\begin{example}\label{ex:D-to-A}
Consider the Coxeter diagrams
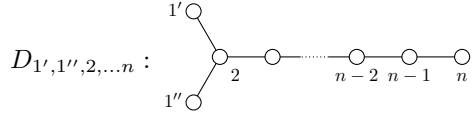
\[
D_{1',1'',2,\ldots n}:  \dynkin[labels={n,n-1,n-2,,2,1',1''}, label directions={,,,,,,right,,},
            scale=2,backwards= true] D{ooo...oooo}
\]
and
\[
A_{1\ldots n}:  \dynkin[labels={1,2,3,,n-1,n},             scale=2] A{ooo...ooo}
\]
The map $
\phi(s_i)=\begin{cases}
t_1&\text{if }i\in \{1',1''\}\\
t_i & \text{if }i\geq2 
\end{cases}
$
defines a morphism of Coxeter systems $\varphi\colon D\to A$.
\end{example}

\begin{example}\label{ex:hatD-to-D}
Consider the Coxeter diagrams
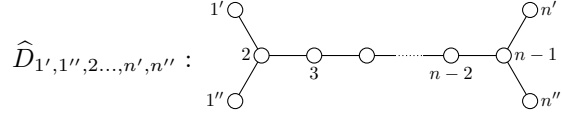
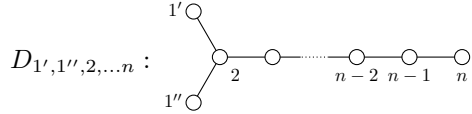
\[
\widehat D_{1',1'',2\ldots, n',n''}: \dynkin[labels={1',1'',2,3,,n-2,n-1,n',n''}, label directions={,,left,,,,right,,},
            scale=2,
            extended] D{oooo...oooo}
\]
and
\[
D_{1',1'',2,\ldots n}:  \dynkin[labels={n,n-1,n-2,,2,1',1''}, label directions={,,,,,,right,,},
            scale=2,backwards= true] D{ooo...oooo}
\]
The map 
$
\phi(s_i)=\begin{cases}
t_n&\text{if }i\in \{n',n''\}\\
t_i & \text{if }i\leq n-1 \text{ or } i\in\{1',1''\}
\end{cases}
$
is a morphism of Coxeter systems $\phi\colon \widehat{D}\to D$.
\end{example}
\begin{example}
By combining Examples \ref{ex:hatD-to-D} and \ref{ex:D-to-A} we get a morphism of Coxeter systems $\hat{D}\to A$.
\end{example}

\begin{example}\label{ex:hatB-to-B}
Consider the Coxeter diagrams
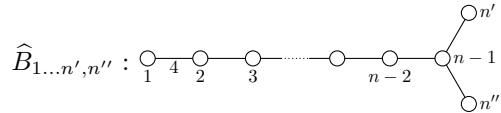
\[
\widehat B_{1\ldots n',n''}: \dynkin[labels={n',n'',n-1,n-2,,3,2,1}, scale=2, extended,Coxeter,backwards= true, label directions={,,left,,,,,,}]B{oooo...ooo}
\]
and
\[
B_{1\ldots n}:  \dynkin[labels={n,n-1,,3,2,1},             scale=2,Coxeter,backwards= true] B{ooo...ooo}
\]
The map
$
\phi(s_i)=\begin{cases}
t_n&\text{if }i\in \{n',n''\}\\
t_i & \text{if }i\leq n-1
\end{cases}
$
is a morphism of Coxeter systems $\phi\colon \widehat{B}\to B$.
\end{example}
 
To produce more elaborate examples we need the following general
result, which describes the centralizer $C_W(r)$  of a simple reflection $r$  in a Coxeter group.
 Let $(W,S)$ be a Coxeter system. For $s\in S$, let $\Gamma^W(s)$ denote the set of roots orthogonal to $\alpha_s$, and for $w\in W$, let $N(w)$ denote the inversion set of $w$, i.e. 
$$N(w)=\{\alpha\in \Phi^+\mid w^{-1}(\alpha)\in -\Phi^+\}.$$

\begin{proposition}\cite[\S2, Theorem]{B}\label{prop:cenralizers}
For every  $r\in S$ one has
$$C_W(r)=W(\Gamma^W(r)\cup \{\pm\alpha_r\})\rtimes Y^W_{r,r},$$
where
$
Y^W_{r,r}=\{u\in W\mid \alpha_r=u(\alpha_{r}),\,N(u)\cap \Gamma^W(r)=\emptyset\}.
$
Moreover $W(\Gamma^W(r)\cup \{\pm\alpha_r\})$ is a  Coxeter group, and $Y^W_{r,r}$ is free of rank $e(r)- n(r)+ 1$, where $e(r)$ denotes the number of edges and $n(r)$ the number of vertices of the connected component of the odd Coxeter graph of $(W,S)$, i.e, the graph with vertex set $S$  and edge set 
$\{\{r, s\} \mid  m_{r,s} \text{ is odd}\}$, containing $r$.
\end{proposition}

\begin{example}\label{ex:B-to-A}
Consider the Coxeter diagrams
\[
B_{1\ldots n}:  \dynkin[labels={n,n-1,,3,2,1},             scale=2,Coxeter,backwards= true] B{ooo...ooo}
\]
and
\[
A_{1\ldots n}:  \dynkin[labels={1,2,3,,n-1,n},             scale=2] A{ooo...ooo}
\]
Then
\begin{itemize}
\item the map $\phi\colon \{s_1,s_2,\dots,s_n\}\to W(A_{1\ldots n})$ defined by
$
\phi(s_i)=\begin{cases}
1&\text{if }i=1\\
t_i & \text{if }i\geq2 
\end{cases}
$
extends to a group homorphism $\phi\colon W(B_{1\ldots n})\to W(A_{1\ldots n})$;
\item the set $\mathcal{S}=\{s_1,s_2\}$ is a set of representatives for the for the equivalence relation on $\{s_1,s_2,\dots,s_n\}$ induced by the conjugation action of $W(B_{1\ldots n})$ on itself;
\item
the pair $(\phi,\mathcal{S})$ satisfies the assumptions of  Lemma \ref{lemma:to-produce-triples}.
\end{itemize}
The first statement is easily verified by checking that  the Coxeter relations are verified, and  the second statement is well-known
(see e.g. \cite{R}). To prove the third statement, notice that $\mathcal{S}_0=\{s_1\}$, so we only need to show that the set
\[
K_{s_1}=\{t \in \{t_1,\dots,t_n\}\text{ such that } \phi(C_{B}(s_1))\subseteq C_{A}(t)\},
\]
is nonempty, where we wrote $C_{B}(s_1)$ for $C_{W(B_{1\ldots n})}(s_1)$ and $C_{A}(t)$ for $C_{W(A_{1\ldots n})}(t)$; we will use similar conventions in the  following. We use Proposition \ref{prop:cenralizers} to determine $C_{B}(s_1)$. The group $Y_{s_1,s_1}^B$ is free of rank 0, so it is the trivial group. We have  
\begin{align*}\Gamma^B(s_1)=&\pm\{\alpha_h+\ldots+\alpha_k, 3\leq h\leq k\leq n\}\\&\cup\pm\{
\sum_{i=1}^na_i\alpha_i\in \Phi(B_{1,\ldots,n})\mid a_2=\sqrt{2} a_1>0\},\end{align*}
which is a system of type $B_{n-1}$. A simple system might be $\{\alpha_1+\sqrt{2} \alpha_2, \alpha_3,\ldots,\alpha_n
\}$. So $C_{B}(s_1)$ is of type $A_1\times B_{n-1}$ and, since $s_{\alpha_1+\sqrt{2} \alpha_2}=s_2s_1s_2$, we have that
\[
C_{B}(s_1)=\langle s_1,s_2s_1s_2, s_3,\dots,s_n\rangle.
\]
Therefore, $
\phi(C_{B}(s_1))=\langle \phi(s_1),\phi(s_2s_1s_2), \phi(s_3),\dots,\phi(s_n)\rangle=\langle s_3,\dots,s_n\rangle.
$
It is clear that $\Gamma^A(s_1)= \{ \pm\alpha_3,\ldots,\pm\alpha_n\}$, and that $Y^A_{s_1,s_1}$ is the trivial group (being free of rank zero). Hence 
$C_{A}(t_1)$ is the group of type $A_1\times A_{n-2}$ given by $\langle t_1,t_3,\dots,t_n\rangle$. This gives $\phi(C_{B}(s_1))\subseteq C_{A}(t_1)$ and so $t_1\in K_{s_1}$. 
\end{example}

\begin{example}
A set of representatives for the equivalence relation on $\{s_1,s_2,\dots,$ $s_{n'},s_{n''}\}$ induced by the conjugation action of $W(\widehat{B}_{1,2,\dots,n',n''})$ on itself is given by $\{s_1,s_2\}$. A set of representatives for the equivalence relation on $\{t_1,t_2,\dots,t_{n}\}$ induced by the conjugation action of $W(B_{1,2,\dots,n})$ on itself is given by $\{t_1,t_2\}$. The morphism of Coxeter systems from Example \ref{ex:hatB-to-B} maps $\{s_1,s_2\}$ to $\{t_1,t_2\}$. Hence, applying Lemma \ref{lem:composition} to Examples \ref{ex:hatB-to-B} and \ref{ex:B-to-A} we get a pair $(\phi,\mathcal{S})$ with
$
\phi\colon W(\widehat{B}_{1,2,\dots,n',n''})\to W(A_{1,\dots,n}),
$
satisfying the assumptions of  Lemma \ref{lemma:to-produce-triples}.
\end{example}

By flipping Example \ref{ex:B-to-A} we obtain the following.
\begin{example}\label{ex:C-to-A}
Consider the Coxeter diagrams
\[
C_{1\ldots n}:  \dynkin[labels={1,2,3,,n-1,n},             scale=2,Coxeter] C{ooo...ooo}
\]
and
\[
A_{1\ldots n}:  \dynkin[labels={1,2,3,,n-1,n},             scale=2] A{ooo...ooo}
\]
Then
\begin{itemize}
\item the map $\phi\colon \{s_1,s_2,\dots,s_n\}\to W(A_{1\ldots n})$ defined by
$
\phi(s_i)=\begin{cases}
1&\text{if }i=n\\
t_i & \text{if }i\leq n-1 
\end{cases}
$
extends to a group homorphism $\phi\colon W(C_{1\ldots n})\to W(A_{1\ldots n})$;
\item the set $\mathcal{S}=\{s_{n-1},s_n\}$ is a set of representatives for the for the equivalence relation on $\{s_1,s_2,\dots,s_n\}$ induced by the conjugation action of $W(C_{1\ldots n})$ on itself;
\item
the pair $(\varphi,\mathcal{S})$ satisfies the assumptions of  Lemma \ref{lemma:to-produce-triples}.
\end{itemize}
\end{example}

\begin{example}\label{CC}
Consider the Coxeter diagrams
\[
\widehat C_{1\ldots n}: \dynkin[labels={1,2,3,,n-2,n-1,n}, scale=2, extended,Coxeter]C{ooo...ooo}
\]
and
\[
C_{1\ldots n}:  \dynkin[labels={1,2,3,,n-1,n},             scale=2,Coxeter] C{ooo...ooo}
\]
Then
\begin{itemize}
\item the map $\phi\colon \{s_1,s_2,\dots,s_n\}\to W(C_{1\ldots n})$ defined by $
\phi(s_i)=\begin{cases}
1&\text{if }i=1\\
t_i & \text{if }i\geq2 
\end{cases}$
extends to a group homorphism $\varphi\colon W(\widehat{C}_{1\ldots n})\to W(C_{1\ldots n})$;
\item the set $\mathcal{S}=\{s_1,s_2,s_n\}$ is a set of representatives for the for the equivalence relation on $\{s_1,s_2,\dots,s_n\}$ induced by the conjugation action of $W(\widehat{C}_{1\ldots n})$ on itself;
\item
the pair $(\varphi,\mathcal{S})$ satisfies the assumptions of  Lemma \ref{lemma:to-produce-triples}.
\end{itemize}
The first two statements are checked as in Example \ref{ex:B-to-A}.
To prove the third statement, notice that $\mathcal{S}_0=\{s_1\}$, so we only need to show that the set
\[
K_{s_1}=\{t \in \{t_1,\dots,t_n\}\text{ such that } \phi(C_{\widehat{C}}(s_1))\subseteq C_{C}(t)\},
\]
is nonempty. We use Proposition \ref{prop:cenralizers} to determine $C_{\widehat{C}}(s_1)$. The group $Y_{s_1,s_1}^{\widehat{C}}$ is free of rank 0, and so it is the trivial group. We have  
$\Gamma^{\widehat{C}}(s_1)=\Phi(C_{3,\ldots,n}).$
Next we determine $C_{C}(t_1)$. The group $Y_{t_1,t_1}^{C}$ is free of rank 0, and so it is the trivial group. Set $\beta=\alpha_1+2(\alpha_2+\ldots+\alpha_{n-1})+\sqrt{2}\alpha_n$. We have 
$\Gamma^{C}(t_1)=\{\pm \beta\}\cup\Phi(C_{3,\ldots,n}),$
which is a root system of type $A_1\times C_{n-3}$. A simple system might be $\{\beta, \alpha_3,\ldots,\alpha_n
\}$. Since $s_{\beta}=s_2s_3\cdots s_{n-1} s_n s_{n-1} \cdots s_3 s_2 s_1 s_2 s_3 \cdots s_{n-1} s_n s_{n-1} \cdots s_2$, we have that
\begin{align*}
\phi(s_{\beta})&=t_2t_3\cdots t_{n-1} t_n t_{n-1} \cdots t_3 t_2  t_2 t_3 \cdots t_{n-1}t_n t_{n-1} \cdots t_2= 1\\
\phi(s_i)&=t_i,\quad 3\leq i\leq n.
\end{align*}
This gives $\phi(C_{\widehat{C}}(s_1))\subseteq C_{C}(t_1)$ and so $t_1\in K_{s_1}$.
\end{example}

\begin{example}
Applying lemma \ref{lem:composition} to Examples \ref{CC} and \ref{ex:C-to-A} we get a pair $(\phi,\mathcal{S})$ with
$
\phi\colon W(\widehat{C}_{1,2,\dots,n,n'})\to W(A_{1,\dots,n}),
$
satisfying the assumptions of  Lemma \ref{lemma:to-produce-triples}.
\end{example}

\section{Braids and ties algebras}
Given a Juyumaya map $(\phi,e)$ between the Coxeter systems $X=(W,S)$ and $Y=(U,T)$, we see that the algebra $\mathcal{E}^{(W,S)}(R[\Sigma(\Phi_U)_{[e]}],\rho_\phi,\hat{e})$ is generated by elements $g_s$ with $s\in S$ and $w. e_s$ with $w\in W$. When  $X$ and $Y$ are of $A$-, $B$-, $C$- or $D$-type, we can interpret the elements $g_s$ as braids and the elements $e_s$ as ties. Consequently the algebra $\mathcal{E}^{(W,S)}(R[\Sigma(\Phi_U)_{[e]}],\rho_\phi,\hat{e})$ will be called a \emph{braids and ties algebra of type $X$-$Y$}. When $Y=A$ we will simply say braids and ties algebra of type $X$.  Somewhat  abusing terminology, for the braids and ties algebras arising from the examples in Section 3, with the exception of the terminal morphism from Example \ref{ex:terminal}, we will say ``the'' braids and ties algebra of type $X$, to mean the distinguished braids and ties algebra of type $X$ arising from those examples.
\begin{example}[The Hecke algebra ``doubled"] This is the case of the terminal morphisms from Example \ref{ex:terminal}. \par If $X=X_{i_1,\dots,i_n}$, the algebra $\mathcal{E}^{(W,S)}(R[\Sigma(\Phi_U)_{[e]}],\rho_\phi,\hat{e})$ in this case is generated by the elements $g_{s_i}$ with $i\in \{i_1,\dots,i_n\}$ and by the single element $e_{s_{i_1}}$  (since we have $e_{s_i}=e_{s_j}$ for any $i,j\in \{i_1,\dots,i_n\}$). Writing $g_i$ for $g_{s_i}$ and $e$ for $e_{s_1}$, the relations defining the algebra are
\begin{align*}
\underbrace{g_{s_i}g_{s_j}g_{s_i} \cdots}_\text{$m_{i,j}$ times}&=\underbrace{g_{s_j}g_{s_i}g_{s_j} \cdots}_\text{$m_{i,j}$ times}\\
 {g_i}^2 &= 1+(u_i-1)e(1-g_i)\\
g_i e &=e g_{i}\\
e^2&=e,
\end{align*}
where $u\colon \{i_1,\dots,i_n\}\to R^\times$ is constant over $W(X_{i_1,\dots,i_n})$-orbits in $\{i_1,\dots,i_n\}$.
For instance, when $X$ is the Coxeter diagram
$
G_{1,2}: \dynkin [Coxeter,gonality=6]G2
$
we get the algebra over three generators $g_1,g_2,e$ with the relations
\begin{align*}
&g_1g_2g_1g_2g_1g_2=g_2g_1g_2g_1g_2g_1\\
 &{g_1}^2 = 1+(u-1)e(1-g_1),\quad
 {g_2}^2 = 1+(u-1)e(1-g_2)\\
&g_1 e =e g_1,\quad
g_2e=eg_2\\
&e^2=e,
\end{align*}
which has manifestly $\{\epsilon\, w\mid w\in W(G_{1,2}),\,\epsilon\in\{1, e\}\}$ as basis over the ground field.
\end{example}
\begin{example}[Marin's $\mathcal C_W$-algebras] We recover here the  Example \ref{ex:marin-roots}. It indeed corresponds to the general identity morphism from Example \ref{ex:A-to-A}. If $X=(W,S)$  is a Coxeter system and $e: S\to \Sigma(\Phi_W) $, as in \eqref{e},  is the map 
$e(s)=\langle s \rangle$, then 
$\mathcal{E}^{(W,S)}(R[\Sigma(\Phi_W)_{[e]}], \rho_{Id},\hat{e})$ is the algebra  $\mathcal C_W$. In particular, when $R=k$ then $\mathrm{rk}_k R[\Sigma(\Phi_U)_{[e]}]$ is the number of reflection subgroups of $W$, which has been called {\it Bell number} of $W$ in \cite{marin}. Hence, formula \eqref{rk} specializes to Marin's rank formula from \cite[Theorem 3.4]{marin}: $\mathrm{rk}_k \mathcal C_W=Bell(W)|W|$.

\end{example}
\begin{example}[The type $A$ braids and ties algebra]\label{ex:bt-a-to-A} This is  a special case of the previous example, when $W$ is of type $A_{n-1}$.  The algebra $\mathcal{E}^{(W,S)}(R[\Sigma(\Phi_U)_{[e]}],\rho_\phi,\hat{e})$ in this case is generated by the elements $g_{s_i}$ with $i\in \{1,\dots,n-1\}$ and by the elements $w. e_{s_i}$ with $w\in W(A_{1,\dots,n-1})\simeq S_{n}$. Let us write $g_i$ for $g_{s_i}$ and $e_i$ for $e_{s_i}$. 

In \cite{AY} by Aicardi an Juyumaya introduced the (type $A$) {\it algebra of braids and ties}. This is the algebra $\mathcal E_n$ generated by the elements $\{g_1,\ldots,g_{n-1},e_1,
\ldots,e_{n-1}\}$ subject to the following relations 
\begin{align}
\label{aa1} g_ig_j&=g_jg_i \quad\text{if}\quad |i-j|>1\\
\label{aa2}g_{i}g_{j}g_{i}&=g_{j}g_{i}g_{j} \quad\text{if}\quad |i-j|=1 \\
\label{aa3} {g_i}^2 &= 1+(u-1)e_i(1-g_i)\\
\label{aa33}g_ie_j&=e_jg_i\qquad \text{if}\quad |i- j|\neq 1\\
\label{aa4} g_jg_ie_j&=e_ig_jg_i \qquad \text{if}\quad |i-j|= 1\\
\label{aa5}e_ie_j&=e_je_i\\
\label{aa6}{e_i}^2&=e_i\\
\label{aa7} e_ie_jg_j&=e_ig_je_i=g_je_ie_j\qquad \text{if}\quad |i-j|= 1
\end{align}
The algebra $\mathcal E_n$ has the following useful and important diagrammatic realization, proved in \cite{AY}. Consider the algebra generated by concatenation\footnote{$AB$ means $A$ on top of $B$.} of the following elementary diagrams on $n$ strands: $1$ is represented by 
 \newcount\beziercnt \beziercnt = 100   

 \newcount\grcalca
 \newcount\grcalcb
 \newcount\grcalcc
 \newcount\grcalcd
 \newcount\mordiam
 \newcount\radius
 \newcount\parametr
 \newcount\breit
\newcount\width

 \mordiam = 8   
 \radius = 3    
 \parametr = 0  
 \newcommand\thlines{\thinlines} 

 \newcommand{\braid}{
  \divide \parametr by -3
  \advance \parametr by 10                                 
   \thlines                                                             
   \bezier{\beziercnt}(\parametr,10)(10,7.5)(7,6)
   \bezier{\beziercnt}(0,0)(0,2.5)(3,4)                 
   \bezier{\beziercnt}(0,10)(0,7.5)(5,5)               
 \bezier{\beziercnt}(\parametr,0)(10,2.5)(5,5)
            \parametr = 0 }                       

\newcommand{\zero}[2]{
                    \grcalca = #1
                    \put(\grcalca,#2){\line(0,1){10}}
                   \advance \grcalca by 10
                    \put(\grcalca,#2){\line(0,1){10}}
                    }

\newcommand{\zerr}[2]{
                    \grcalca = #1
                    \advance \grcalca by 10
                    \put(\grcalca,#2){\line(0,1){10}}
                    }

\newcommand{\zerl}[2]{
                    \grcalca = #1
                    \put(\grcalca,#2){\line(0,1){10}}
                   }

\newcommand{\ei}[2]{                    \grcalca = #1
                    \grcalcb = #2
                     \put(\grcalca,\grcalcb){\line(0,1){10}}

                    \advance \grcalca by 10
                    \advance \grcalcb by 5

                    \put(\grcalca,#2){\line(0,1){10}}
                    \put(#1,\grcalcb){\dashbox{1}(10,0)}
                    }

\newcommand{\pzero}[2]{
                    \grcalca = #1
                    \put(\grcalca,#2){\line(0,1){2}}
                   \advance \grcalca by 10
                    \put(\grcalca,#2){\line(0,1){2}}
                    }

\newcommand{\pzerr}[2]{
                    \grcalca = #1
                    \advance \grcalca by 10
                    \put(\grcalca,#2){\line(0,1){2}}
                    }

\newcommand{\pzerl}[2]{
                    \grcalca = #1
                    \put(\grcalca,#2){\line(0,1){2}}
                   }

\newcommand{\pei}[2]{                    \grcalca = #1
                    \grcalcb = #2
                     \put(\grcalca,\grcalcb){\line(0,1){2}}

                    \advance \grcalca by 10
                    \advance \grcalcb by 1

                    \put(\grcalca,#2){\line(0,1){2}}
                    \put(#1,\grcalcb){\dashbox{1}(10,0)}
                    }

 \newcommand{\ibraid}{
 \divide \parametr by -3
  \advance \parametr by 10                         
   \thlines                                        
   \bezier{\beziercnt}(\parametr,10)(10,7.5)(5,5)  
   \bezier{\beziercnt}(0,0)(0,2.5)(5,5)            
   \bezier{\beziercnt}(0,10)(0,7.5)(3,6)           
   \bezier{\beziercnt}(\parametr,0)(10,2.5)(7,4)
     \parametr = 0 }                               

\newcommand{\sibrmor}{\brmor}
\newcommand{\isibrmor}{\ibrmor}
\newcommand{\simor}{\brmor}
\newcommand{\isimor}{\ibrmor}


\newcommand{\nubrmor}{
\thlines                                 
   \bezier{\beziercnt}(9.8,10)(10,8.5)(6,6)  
   \bezier{\beziercnt}(0,0)(0,1.5)(4,4)     
   \bezier{\beziercnt}(0,10)(0,8.5)(3,7)    
   \bezier{\beziercnt}(9.8,0)(10,1.5)(7,3)
\bezier{\beziercnt}(4,4)(4,4)(6,6)
\bezier{\beziercnt}(3,7)(7,7)(7,3)\bezier{\beziercnt}(3,7)(3,3)(7,3)}

\newcommand{\inubrmor}{ 
\thlines
  \bezier{\beziercnt}(4,6)(4,6)(6,4)
   \bezier{\beziercnt}(9.8,10)(10,8.5)(7,7)  
   \bezier{\beziercnt}(0,0)(0,1.5)(3,3)     
   \bezier{\beziercnt}(0,10)(0,8.5)(4,6)    
   \bezier{\beziercnt}(9.8,0)(10,1.5)(6,4)
\bezier{\beziercnt}(3,3)(3,7)(7,7)
\bezier{\beziercnt}(3,3)(7,3)(7,7)}

\newcommand{\rmor}[1]{
\put(0,10){\twist{2}{-4}}
\put(10,10){\twist{-2}{-4}}
\put(5,4){\circle{8}}
\put(5,4){\makebox(0,0){$#1$}}}

 \newcommand{\dualmor}{                      
 \thlines                                    %
 \put(-8,5){\oval(6,6)[l]}                   
 \bezier{\beziercnt}(-.2,10)(0,8)(-4,4)       %
 \bezier{\beziercnt}(-4,4)(-6,2)(-8,2)
\bezier{\beziercnt}(-8,8)(-6,8)(-5,7)        %
 \bezier{\beziercnt}(-1,3)(0,2)(-.2,0)}     %

\newcommand{\dualmors}{                      
 \thlines                                    %
 \put(-8,5){\oval(6,6)[l]}                   
 \bezier{\beziercnt}(-.2,0)(0,2)(-4,6)       %
 \bezier{\beziercnt}(-4,6)(-6,8)(-8,8)
\bezier{\beziercnt}(-8,2)(-6,2)(-5,3)        %
 \bezier{\beziercnt}(-1,7)(0,8)(-.2,10)}     %

 \newcommand{\idualmor}{                      
 \thlines                                     %
 \put(8,5){\oval(6,6)[r]}                     
 \bezier{\beziercnt}(0,10)(0,8)(4,4)           %
 \bezier{\beziercnt}(4,4)(6,2)(8,2)
\bezier{\beziercnt}(8,8)(6,8)(5,7)            %
 \bezier{\beziercnt}(1,3)(0,2)(0,0)}         %

 \newcommand{\idualmors}{                      
 \thlines                                     %
 \put(8,5){\oval(6,6)[r]}                     
 \bezier{\beziercnt}(0,0)(0,2)(4,6)           %
 \bezier{\beziercnt}(4,6)(6,8)(8,8)
\bezier{\beziercnt}(8,2)(6,2)(5,3)            %
 \bezier{\beziercnt}(1,7)(0,8)(0,10)}         %

\newcommand{\dualbrmors}{\dualmors\put(-3,5){\circle{6}}}
\newcommand{\idualbrmors}{\idualmors\put(3,5){\circle{6}}}
\newcommand{\dualbrmor}{\dualmor\put(-3,5){\circle{6}}}
\newcommand{\idualbrmor}{\idualmor\put(3,5){\circle{6}}}

 \newcommand{\mor}[2]{
   \thlines
   \grcalca = #2
   \advance \grcalca by -\radius
 \advance \grcalca by -\radius
   \divide \grcalca by 2                               
   \put(0,#2) {\line(0,-1){\grcalca}}             
   \put(0,0) {\line(0,1){\grcalca}}                
   \grcalca = #2                                        
   \divide \grcalca by 2                           
   \put(0,\grcalca) {\circle{\mordiam}}     
   \put(0,\grcalca) {\makebox(0,0){$#1$}} } 

\newcommand{\umor}[2]{                          %
 \thlines                                                         
 \grcalca = #2                                                
\advance \grcalca by -\radius
 \advance \grcalca by -\radius
\put(0,\grcalca) {\line(0,-1){\grcalca}}       
\advance \grcalca by \radius
\put(0,\grcalca) {\circle{\mordiam}}
\put(0,\grcalca) {\makebox(0,0){$#1$}} }        
\newcommand{\coumor}[2]{                          %
 \thlines                                                          
 \grcalca = #2                                                  
     \advance \grcalca by -\radius
 \advance \grcalca by -\radius
\put(0,6) {\line(0,1){\grcalca}}      
\put(0,\radius) {\circle{\mordiam}}             
\put(0,\radius) {\makebox(0,0){$#1$}} }         

  \newcommand{\multmor}[3]{                    
    \thlines                                   
    \grcalca = #2
    \grcalcb = #3                              
    \advance \grcalca by 5                     
    \put(-2.5,0){\framebox(\grcalca,\grcalcb)  
     {$#1$}}}

 \newcommand{\twist}[2]{                     
   \thlines                                               
   \grcalca = #1                                      
   \divide \grcalca by 2                          
   \grcalcb = #2                                      
   \divide \grcalcb by 2                          
   \grcalcc = #2                                      
   \divide \grcalcc by 3                          
   \bezier{\beziercnt}(0,0)(0,\grcalcc)(\grcalca,\grcalcb)
   \grcalcd = #2
   \advance \grcalcd by -\grcalcc
   \multiply \grcalcc by 3
   \bezier{\beziercnt}(\grcalca,\grcalcb)(#1,\grcalcd)(#1,#2) }

 \newcommand{\idgr}[1]{                     
   \thlines                                 
   \put(0,0) {\line(0,1){#1}}}              

  \newcommand{\rigco}[1]{                   
    \thlines                                
    \put(0,0) {\line(1,0){#1}}              
    \put(#1,0) {\line(-5,2){#1}}}           

\newcommand{\ig}[1]{\idgr{#1}}

  \newcommand{\lefco}[1]{                   
    \thlines                                
    \put(#1,0) {\line(0,1){5}}              
    \put(0,0) {\line(1,0){#1}}              
    \put(0,0) {\line(5,2){#1}}}             

  \newcommand{\rigmu}[1]{                   
    \thlines                                
    \put(0,0) {\line(0,1){5}}               
    \put(0,5) {\line(1,0){#1}}              
    \put(#1,5) {\line(-5,-2){#1}}}          

  \newcommand{\lefmu}[1]{                   
    \thlines                                
    \put(#1,0) {\line(0,1){5}}              
    \put(0,5) {\line(1,0){#1}}              
    \put(0,5) {\line(5,-2){#1}}}            

\newcommand{\objo}[1]{
   \put(0,3){\makebox(0,0)[b]{\mbox{$#1$}}}}

 \newcommand{\oo}{\objo}

 \newcommand{\obju}[1]{
   \put(0,-5){\makebox(0,0)[b]{\mbox{$#1$}}}} 
\newcommand{\ou}{\obju}

\newcommand{\bgr}[3]{
   \unitlength=#1 
\divide \unitlength by 2
\advance  \unitlength by #1
 \grcalca = #3
   \advance \grcalca by 6
 \breit = #2
\advance \breit  by 5
   \begin{picture}(#2,\grcalca)}

  \def\eeee{\put(\breit,0){\obju{\mbox{\normalsize\rule{0.75ex}{1.5ex}}}}}
\thlines

\newcommand{\egr}{\end{picture}}

\newcommand{\fbgr}[2]{\scriptsize\bgr{.4ex}{#1}{#2}}
\newcommand{\fegr}{\egr\nml
}
\newcommand{\cbgr}[2]{\scriptsize\begin{center}\bgr{.4ex}{#1}{#2}}
\newcommand{\cegr}{\egr\end{center}\nml\vspace{1ex}
}

 \newcommand{\bsgr}[1]{
   \newsavebox{#1}
   \savebox{#1}(0,0)[bl]}

\setlength\unitlength{0.4ex}
\begin{picture}(100,35)(-40,5)

\bgr{0.4ex}{100}{20}

\zerl{10}{10}
\zerl{20}{10}
\zerl{30}{10}
\zerl{50}{10}
\zero{60}{10}
\put(38,15){$\ldots$}
\scriptsize
\put(10,21){$1$}
\put(20,21){$2$}

\put(30,21){$3$}
\put(48,21){$n-2$}
\put(58,21){$n-1$}
\put(70,21){$n$}
\egr
\end{picture}

\noindent $g_i$ is represented by a elementary {\it braid}

\begin{picture}(100,35)(-40,5)

\bgr{0.4ex}{100}{20}

\zero{10}{10}
\put(40,10)\ibraid

\zerl{70}{10}
\put(28,15){$\ldots$}
\put(58,15){$\ldots$}

\scriptsize

\put(10,21){$1$}
\put(20,21){$2$}
\put(40,21){$i$}
\put(48,21){$i+1$}
\put(70,21){$n$}
\egr
\end{picture}

\noindent $e_i$ is represented by a elementary {\it tie}

\begin{picture}(100,35)(-40,5)

\bgr{0.4ex}{100}{20}

\zero{10}{10}
\ei{40}{10}

\zerl{70}{10}
\put(28,15){$\ldots$}
\put(58,15){$\ldots$}

\scriptsize
\put(10,21){$1$}
\put(20,21){$2$}
\put(40,21){$i$}
\put(48,21){$i+1$}
\put(70,21){$n$}
\egr
\end{picture}

The resulting diagrams are considered up to planar isotopy; moreover, the endpoints of the ties can freely move along the strand and through the ties themselves.
Our general constructions recovers this algebra, as proved in the following
\begin{proposition}
    The algebra $\mathcal{E}^{(W,S)}(R[\Sigma(\Phi_U)_{[e]}],\rho_\phi,\hat{e})$ coincides with $\mathcal E_n$.
\end{proposition}
 A (sketchy) proof of this statement can be found in \cite{marin}. We provide a detailed treatment, which can be taken as a basis for other instances of braids and type algebras admitting a geometric realization.
\begin{proof}
The relations defining $\mathcal{E}^{(W,S)}(R[\Sigma(\Phi_U)_{[e]}],\rho_\phi,\hat{e})$ are 
\begin{align}
\label{a1} g_ig_j&=g_jg_i \quad\text{if}\quad |i-j|>1\\
\label{a2} g_{i}g_{j}g_{i}&=g_{j}g_{i}g_{j} \quad\text{if}\quad |i-j|=1 \\
\label{a3}{g_i}^2 &= 1+(u-1)e_i(1-g_i)\\
\label{a4}g_i (w. e_j) &=(s_iw. e_j)g_{i}, \qquad\forall w\in W(A_{1,\dots,n-1})\\
\label{a5}(w_1. e_i)(w_2. e_j)&=(w_2. e_j)(w_1. e_i), \qquad\forall w_1, w_2\in W(A_{1,\dots,n-1})\\
\label{a6}{(w. e_i)}^2&=w. e_i , \qquad\forall w\in W(A_{1,\dots,n-1})\\
\label{a7} e_i (w. e_j)  &=e_i ((s_iw). e_j)\quad \qquad\forall w\in W(A_{1,\dots,n-1}),
\end{align}
where the fact that all of the $s_i$'s are in the same $W$-orbit implies that the datum of the $W$-equivariant map $u\colon\{s_1,\dots,s_n\}\to R^\times$ reduces to that of the choice of an element $u$ in $R^\times$.
By Remark \ref{rem:invertibility} and Example \ref{example:monoid}, equation ${g_i}^2 = 1+(u_i-1)e_i(1-g_i)$ implies that $g_i$ is invertible, with inverse given by a polynomial expression in $g_i$ and $e_i$. So we can rewrite equation $g_i (w. e_j) =(s_iw. e_j)g_{i}$ as
\[
s_iw. e_j=g_i (w. e_j){g_i}^{-1}
\]
and argue by induction on the lenght of $w$ that all the elements $w. e_i$ lie in the subalgebra generated by the $g_j$'s and the $e_j$'s. Therefore $\{g_1,\ldots,g_{n-1},e_1,
\ldots,e_{n-1}\}$ is a set of generators for the algebra $\mathcal{E}^{(W,S)}(R[\Sigma(\Phi_U)_{[e]}],\rho_\phi,\hat{e})$ in this case.
Taking $w=1$, relation \eqref{a4} gives 
\begin{equation}\label{eeee}
g_ie_j=(s_i. e_j)g_i.
\end{equation}
  Since we have
\begin{equation}\label{eq:either-or}
 s_i. e_j=\begin{cases}
     e_j&\text{if}\quad |i-j|\neq 1\\
     s_j. e_i&\text{if}\quad |i-j|= 1,
 \end{cases}   
\end{equation}
equation
\eqref{eeee}  reduces to 
$
g_ie_j=e_jg_i\qquad \text{if}\quad |i- j|\neq 1,$
which is relation \eqref{aa33}.

Moreover, if $|i-j|=1$,  using Remark \ref{rem:commutes}, we have 
\begin{align*}
g_i e_j=(s_i. e_j)g_{i}&=(s_j. e_i)g_{i}
={g_j}e_ig_{j}^{-1}g_{i}
={g_j}^{-1}e_ig_jg_{i},
\end{align*}
i.e.,
$
g_jg_ie_j=e_ig_jg_i \qquad \text{if}\quad |i-j|= 1,
$
which is relation \eqref{aa4}.\par
Taking $w=w_1=w_2=1$, relation \eqref{a6} becomes 
$
{e_i}^2=e_i,
$
which is relation \eqref{aa6}, and relation \eqref{a5} becomes
$
e_ie_j=e_je_i.  
$
which is relation \eqref{aa5}. \par Finally, taking $w=1$ relation \eqref{a7} gives
$
e_i e_j  =e_i (s_i. e_j).$
This is trivially true if $|i-j|\neq 1$, while for $|i-j|=1$, by \eqref{eq:either-or} and \eqref{a5}, it becomes
\[
e_i e_j=e_i(s_j. e_i)=(s_j. e_i)e_i,
\]
i.e.,
$
e_ie_j=e_ig_je_i{g_j}^{-1}
$
and
$
e_ie_j=g_je_i{g_j}^{-1}e_i={g_j}^{-1}e_i{g_j}e_i,
$
i.e., since $e_ie_j=e_je_i$, 
$
e_je_ig_j=e_ig_je_i=g_je_ie_j.
$
Hence relation \eqref{aa7} holds and  since relations \eqref{a1}-\eqref{a3} and \eqref{aa1}-\eqref{aa3} are identical, $\mathcal{E}^{(W,S)}(R[\Sigma(\Phi_U)_{[e]}],\rho_\phi,\hat{e})$  satisfies all the defining relations of $\mathcal E_n$.
Also the converse holds. 
To get \eqref{a4}, notice that it is equivalent to $g_i(w. e_j) g_i^{-1}=(s_iw. e_j)$, which in turn is equivalent to $s_i.(w. e_j)=(s_iw. e_j)$, which clearly holds. Since $w.-$ is an action on a monoid, \eqref{aa6} implies ${(w. e_i)}^2=w. e_i$ for any $w$, i.e., \eqref{a6} holds. By the same reason, equation \eqref{a5}  is equivalent to \begin{equation}\label{daprovare}
e_i(w. e_j)=(w. e_j)e_i,\qquad \forall w\in W(A_{1,\dots,n-1}).\end{equation}
To prove \eqref{daprovare} and, finally, \eqref{a7}, we  resort  to the diagrammatic encoding of $\mathcal E_n$. Denote by $e_{ij}$ the diagram with a tie joining strands $i$ and $j$. The key remark is that the $W$-action on ties can be interpreted as the permutation action of the endpoints of the ties: 
the diagram of $w. e_i$ is $e_{w(i), w(i+1)}$.

Relations \eqref{daprovare},  \eqref{a7} in $\mathcal E_n$ read, respectively    \begin{align}\label{read}e_{i,i+1}e_{w(j), w(j+1)}&=e_{w(j), w(j+1)}e_{i,i+1},\\
\label{reads}e_{i,i+1}e_{w(j), w(j+1)}&=e_{i,i+1}e_{s_iw(j), s_iw(j+1)}.\end{align} 
The are several cases to analyze. In the following table we display the values of $w(j), w(j+1)$ and the corresponding  values of  $s_iw(j), s_iw(j+1)$; by $*$ we mean an integer in the range $1, \ldots, n$ different from $i,i+1$.
\begin{center}
\begin{tabular}{ c |  c | c | c | c | c | c}
 & $1$  & $2$  & $3$  & $4$  & $5$  & $6$ \\\hline
 $w(j)$ &  $i$ & $i+1$ & $i$ & $i+1$ & $*$ & $*$  \\\hline
  $w(j+1)$ & $*$ & $*$ & $i+1$ & $i$ & $i$ & $i+1$ \\\hline
   $s_iw(j)$ &$i+1$ & $i$ & $i+1$ & $i$ & $*$ & $*$  \\\hline
    $s_iw(j+1)$ & $*$ & $*$ & $i$ & $i+1$ & $i+1$ & $i$ 
\end{tabular}
\end{center}
Relation \eqref{read} in cases 3,4 amounts to \eqref{aa5}, where in the other cases follows from the fact that the endpoints move freely on the strands (which easily follows from \eqref{a4}).
Relation \eqref{reads} is clear in cases 3,4, and in cases 1, 2 (and by the same reason also in cases 5,6) corresponds to the following diagrammatic equality between strands 
$i,i+1,*$
\vskip5pt
\begin{picture}(100,25)(-55,5)

\bgr{0.4ex}{10}{18}

\ei{10}{0}

\ei {10}{10}

\zerl{30}{10}

\zerl{40}{0}
\zerl{60}{10}
\ei {20}{0}

\ei {40}{10}

\ei{50}{0}


 \egr
\end{picture}

\end{proof}
\end{example}
\begin{remark}\label{2p}  In \cite{AY4} the authors constructed a two-parameters braids and ties algebra $\mathcal E_n(u,v)$, which specializes to $\mathcal E_n$, and which is relevant for applications to knot theory. We can recover this algebra starting from the generalized Marin ring described in Remark \ref{2var} by setting:
$$b_s=1+(u-1)e_s,\quad c_s=1-(v-1)e_s.$$  With this specialization relation \eqref{H2} (which replaces \eqref{a3}) becomes
\begin{equation}\label{newq}g_s^2=1+(u-1)e_s+(v-1)e_sg_s,\end{equation}which is formula (20) from \cite{AY4}.\par
To obtain  $\mathcal E_n$  from $\mathcal E_n(u,v)$, one looks for rescalings $g_s\rightsquigarrow  \tilde g_s=\lambda_s g_s$ which turn \eqref{newq} to \eqref{a3}. The equation for $\lambda_s$ becomes
\[
(1+(u-1)e_s)\lambda_s^2 -(v-1)e_s\lambda_s-1=0.
\]
Now we use the idempotency of $e_s$ and write
\[
\xi_s=\lambda_se_s; \qquad z_s=\lambda_s(1-e_s),
\]
so that
$
\lambda_s=\xi_s+z_s.
$
By the commutativity of the algebra we have $\xi_sz_s=z_s\xi_s=0$ and so
$
\lambda_s^2=\xi_s^2+z_s^2.
$
The quadratic equation for $\lambda_s$ becomes
\[
(1+(u-1)e_s)(\xi_s^2+z_s^2) -(v-1)e_s(\xi_s+z_s)-1=0,
\]
and so, recalling that $e_s(1-e_s)=0$, it becomes the equation
\[
(1+(u-1)e_s)(\xi_s^2+z_s^2) -(v-1)e_s\xi_s-1=0.
\]
By multiplying it by $e_s$ and by $1-e_s$, this is in turn equivalent to the pair of equations
\[
\begin{cases}
(u\xi_s^2 -(v-1)\xi_s-1)e_s=0\\
(z_s^2 -1)(1-e_s)=0.
\end{cases}
\]
Let us consider the system of equations
\[
\begin{cases}
u\eta_s^2 -(v-1)\eta_s-1=0\\
\zeta_s^2 -1=0.
\end{cases}
\]
If $(\eta_s,\zeta_s)$ is a solution, then $\xi_s=\eta_se_s$, $z_s=\zeta_s(1-e_s)$ is a solution of the previous one. A solution of $\zeta_s^2-1=0$ is clearly given by $\zeta_s=1$. Let then $\tau$ be a solution of the quadratic equation 
$
ut^2 -(v-1)t-1=0
$
(notice that this equation is independent of $s$).
We find that a solution $\lambda_s$ to our original problem is given by
\[
\lambda_s=(1-e_s)+\tau e_s=1+(\tau-1)e_s.
\]
If we set
$
\delta=\tau-1,
$
then our expression for $\lambda_s$ takes the form
$
\lambda_s=1+\delta e_s
$
so that our rescaling is
$
 \tilde{g}_s= g_s+\delta e_s g_s.
$
Notice that the quadratic equation satisfied by $\delta$ is
\[
u(z+1)^2 -(v-1)(z+1)-1=0,
\]
i.e., Aicardi-Juyumaya's equation (22) from \cite{AY4}.
\end{remark}

The previous example shows how to realize braids and ties algebras of type $A$, i.e. of type $A$-$A$. Similarly, one can produce algebras of type $B$-$B$, $D$-$D$ and so on. There are however also genuine examples of braids and ties algebras of type $B$ and $D$ (i.e., of type $B$-$A$, $D$-$A$).
\begin{example}[The type $D$ braids and ties algebra] This is the case when $X=D$ and $Y=A$ coming from Example \ref{ex:D-to-A}. The algebra $\mathcal{E}^{(W,S)}(R[\Sigma(\Phi_U)_{[e]}],\rho_\phi,\hat{e})$ in this case is generated by the elements $g_{s_i}$ with $i\in \{1',1'',2\dots,n\}$ and by the elements $w. e_{t_j}$ with $j\in\{1,2,\dots,n\}$ and $w\in W(D_{1',1'',\dots,n})$ (or equivalently, since the morphism $W(D_{1',1'',\dots,n})\to W(A_{1,\dots,n})$ is surjective, $w\in W(A_{1,\dots,n})$). It is convenient to introduce elements $e_{s_i}$ with $i\in \{1',1'',2\dots,n\}$ by setting $e_{s_{1'}}=e_{s_{1''}}=e_{t_1}$ and $e_{s_i}=e_{t_i}$ for $i\in\{2,\dots n\}$. Let us write $g_i$ for $g_{s_i}$ and $e_i$ for $e_{s_i}$.  The relations defining the algebra are then obtained as in Example \ref{ex:bt-a-to-A}:
\begin{align*}\label{eq:relations-da2} 
\notag g_ig_j&=g_jg_i \quad\text{if}\quad |i-j|>1\\
\notag g_{i}g_{j}g_{i}&=g_{j}g_{i}g_{j} \quad\text{if}\quad |i-j|=1 \\
\notag {g_i}^2 &= 1+(u-1)e_i(1-g_i)\\
g_ie_j&=e_jg_i\qquad \text{if}\quad |i- j|\neq 1\\
\notag g_jg_ie_j&=e_ig_jg_i \qquad \text{if}\quad |i-j|= 1\\
\notag e_ie_j&=e_je_i\\
\notag{e_i}^2&=e_i\\
\notag e_ie_jg_j&=e_ig_je_i=g_je_ie_j\qquad \text{if}\quad |i-j|= 1\\
\notag e_{1'}&=e_{1''},
\end{align*}
where $|i-j|$ denotes the distance in the Dynkin diagram $D_{1',1'',2,\dots,n}$, so that, e.g., $|1'-1''|=2$.
Notice that one has an evident morphism of algebras from the type $D$ to the type $A$ braids and ties algebra, given by
\[
g_{i}\mapsto \begin{cases}
g_1 &\text{if}\quad i\in\{1',1''\}\\
g_i&\text{if}\quad i\geq 2
\end{cases},\quad 
e_{i}\mapsto \begin{cases}
e_1 &\text{if}\quad i\in\{1',1''\}\\
e_i&\text{if}\quad i\geq 2
\end{cases}\quad .
\]

\end{example}

\begin{example}[The type $B$ braids and ties algebra] This is the case when $X=B$ and $Y=A$ coming from Example \ref{ex:B-to-A}. The algebra $\mathcal{E}^{(W,S)}(R[\Sigma(\Phi_U)_{[e]}],\rho_\phi,\hat{e})$ in this case is generated by the elements $g_{s_i}$ with $i\in \{1,\dots,n\}$ and by the elements $w. e_{s_i}$ with $w\in W(B_{1,\dots,n})$. By Example \ref{ex:B-to-A} we can choose $t_0\colon \mathcal{S}_0=\{s_1\}\to \{t_1,\dots,t_n\}$ to be $t_0(s_1)$. The corresponding map $t\colon \mathcal{S}\to \{t_1,\dots,t_n\}$ from Lemma \ref{lemma:to-produce-triples} is then given by $t(s_i)=t_i$ for $i=1,2$.
According to  Lemma  \ref{lemma:equivariant}, the map 
\[
e\colon\{s_1,\dots,s_n\}\to \Phi(W(A_{1,\dots,n},\{t_1,\dots,t_n\})
\]
is given by
\[
e_{s_i}=\phi(w).\langle t(s_0)\rangle,
\]
where $s=ws_0w^{-1}$ with $s_0\in\{s_1,s_2\}$ and $w\in W(B_{1,\dots,n})$. It then follows
that $e$ is explicitly given by
\[
e_{s_i}=\langle t_i\rangle 
\]
for any $i$. As in the previous example, let us write $g_i$ for $g_{s_i}$ and $e_i$ for $e_{s_i}$. Following \cite{F}, let us also recursively define the elements $f_i\in R[\Sigma_{A_{1,\dots,n},\{t_1,\dots,t_n\}}]_{[e]}$ as $f_1=e_1$ and $f_{i}=g_{i}. f_{i-1}$ for $2\leq i\leq n$. 
The relations defining the algebra $\mathcal{E}^{(W,S)}(R[\Sigma(\Phi_U)_{[e]}],\rho_\phi,\hat{e})$ are then
\begin{align} 
\label{rBi} g_ig_j&=g_jg_i \quad\text{if}\quad |i-j|\neq 1\\
g_{1}g_{2}g_{1}g_{2}&=g_{2}g_{1}g_{2}g_{1}\\
 g_{i}g_{j}g_{i}&=g_{j}g_{i}g_{j} \quad\text{if}\quad |i-j|=1, \{i,j\}\neq\{1,2\} \\
 {g_1}^2 &= 1+(v-1)e_i(1-g_i)\\
{g_i}^2 &= 1+(u-1)e_i(1-g_i),\qquad i\geq 2\\
g_i (w. e_j) &=(s_iw. e_j)g_{i}, \qquad\forall w\in W(B_{1,\dots,n})\\
 (w_1. e_i)(w_2. e_j)&=(w_2. e_j)(w_1. e_i), \qquad\forall w\in W(B_{1,\dots,n})\\
{(w. e_i)}^2&=w. e_i , \qquad\forall w\in W(B_{1,\dots,n})\\
e_i (w. e_j)  &=e_i ((s_iw). e_j)\quad \qquad\forall w\in W(B_{1,\dots,n})\\
f_1&=e_1\\
 g_if_j&=f_jg_i \quad\text{if}\quad  i=1 \quad\text{or}\quad i\geq 2, j\neq i-1,i\\
g_if_{i-1}&=f_{i}g_i\quad 2\leq i\leq n\\
g_if_{i}&=f_{i-1}g_i\quad 2\leq i\leq n\\
e_{i+1}f_i&=e_{i+1}f_{i+1}\quad 1\leq i< n\\
\label{rBf}e_{i+1}f_i&=f_{i+1}f_{i}\quad 1\leq i< n,
\end{align}
where we used the fact that all of the $s_i$'s with the exception of $s_1$ are in the same $W$-orbit to reduce the datum of the $W$-equivariant map $u\colon\{s_1,\dots,s_n\}\to R^\times$ to the pair of elements $v$ and $u$ in $R^\times$. 
The origin of the relations involving the elements $f_i$ is as follows. The relation $f_1=e_1$ and the relation $g_if_{i-1}=f_ig_i$ are the defining relations for the elements $f_i$. Next, 
the algebra $R[\Sigma_{A_{1,\dots,n},\{t_1,\dots,t_n]}]_{[e]}$ is a subalgebra of $R[\Sigma_{A_{1,\dots,n},\{t_1,\dots,t_n]}]$. In this larger algebra, under the canonical isomorphism $W(A_{1,\dots,n})\cong S_{n+1}$ mapping the generator $t_i$ to the permutation $(i,i+1)$, the element $e_i$ corresponds to the subgroup of $S_{n+1}$ generated by the permutation $(i,i+1)$, and the element $f_i$ to the subgroup of $S_{n+1}$ generated by the permutation $(1,i+1)$. Then the identity 
$ g_1f_j=f_jg_1$
for all $j$ is the fact that $g_1$ is in the kernel of the homomorphism $W(B_{1,\dots,n})\to W(A_{1,\dots,n})$ and so it acts trivially on $R[\Sigma_{A_{1,\dots,n},\{t_1,\dots,t_n]}]_{[e]}$. The identity
$
g_if_j=f_jg_i
$
for $|i-j|\neq 1$ with $i\geq 2$ and $j\neq i-1,i$ is the fact that with these assumptions in $S_{n+1}$ we have the identity $(i,i+1)(1,j+1)(i,i+1)=(1,j+1)$. Therefore in the larger algebra $R[\Sigma_{A_{1,\dots,n},\{t_1,\dots,t_n]}]$ we have the relation
$g_if_j=f_jg_i$; this is a relation in $R[\Sigma_{A_{1,\dots,n},\{t_1,\dots,t_n]}]$ among elements in $R[\Sigma_{A_{1,\dots,n},\{t_1,\dots,t_n]}]_{[e]}$ and so it is a relation in $R[\Sigma_{A_{1,\dots,n},\{t_1,\dots,t_n]}]_{[e]}$. Similarly, the relation  $
g_if_i=f_{i-1}g_i
$
is the identity $(i,i+1)(1,i+1)(i,i+1)=(1,i)$ in $S_{n+1}$. Finally, the relation 
$
e_{i+1}f_i=e_{i+1}f_{i+1}
$
expresses the fact that in $S_{n+1}$ the pair of permutations $((i+1,i+2),(1,i+1))$ and $((i+1,i+2),(1,i+2))$ generate the same subgroup; similarly the relation 
$
e_{i+1}f_i=f_{i+1}f_{i}
$
expresses the fact that in $S_{n+1}$ the pair of permutations $((i+1,i+2),(1,i+1))$ and $((1,i+1),(1,i+2))$ generate the same subgroup.

Taking $w=1$, the relation $g_i (w. e_j) =(s_iw. e_j)g_{i}$ gives 
$
g_ie_j=(s_i. e_j)g_i.
$
  Since we have the nontrivial relations
\begin{equation}\label{eq:either-or2}
 s_i. e_j=\begin{cases}
     e_j&\text{if}\quad |i-j|\neq 1 \quad \text{or}\quad (1,j)=(1,2),\\
     s_j. e_i&\text{if}\quad |i-j|= 1, \quad \{i,j\}\neq\{1,2\},\\
     \end{cases}   
\end{equation}
equation
$g_i e_j =(s_i. e_j)g_{i}$ reduces to 
\[
g_ie_j=e_jg_i\qquad \text{if}\quad |i- j|\neq 1\quad \text{or}\quad (1,j)=(1,2)
\]
and gives
\begin{align*}
g_i e_j=(s_i. e_j)g_{i}&=(s_j. e_i)g_{i}
={g_j}e_ig_{j}^{-1}g_{i}
={g_j}^{-1}e_ig_jg_{i},
\end{align*}
where we used Remark \ref{rem:commutes},
i.e.,
\[
g_jg_ie_j=e_ig_jg_i \qquad \text{if}\quad |i-j|= 1, \quad \{i,j\}\neq\{1,2\}.
\]
Taking $w=w_1=w_2=1$, the relation ${(w. e_i)}^2=w. e_i $ becomes 
$
{e_i}^2=e_i
$
and the relation $(w_1. e_i)(w_2. e_j)=(w_2. e_j)(w_1. e_i)$ becomes
$
e_ie_j=e_je_i. 
$\par
Taking $w_1=1$, $w_2=g_jg_{j-1}\cdots g_2$ and $j=1$, the relation $(w_1. e_i)(w_2. e_j)=(w_2. e_j)(w_1. e_i)$ becomes
$
e_if_j=f_je_i.
$\par
Taking $w_1=g_ig_{i-1}\cdots g_2$, $w_2=g_jg_{j-1}\cdots g_2$ and $i=j=1$, the relation $(w_1. e_i)(w_2. e_j)=(w_2. e_j)(w_1. e_i)$ becomes
$
f_if_j=f_jf_i.
$\par
Taking $w=1$ the relation $e_i (w. e_j)  =e_i ((s_iw). e_j)$ gives
$
e_i e_j  =e_i (s_i. e_j).
$
This is trivial if $|i-j|\neq 1$ or $(i,j)=(1,2)$, while for $|i-j|=1$ with $\{i,j\}\neq\{1,2\}$, by \eqref{eq:either-or2} and by $(w_1. e_i)(w_2. e_j)=(w_2. e_j)(w_1. e_i)$, it becomes
$
e_i e_j=e_i(s_j. e_i)=(s_j. e_i)e_i,
$
i.e.,
\[
e_ie_j=e_ig_je_i{g_j}^{-1}
\]
and
$
e_ie_j=g_je_i{g_j}^{-1}e_i={g_j}^{-1}e_i{g_j}e_i,
$
i.e., since $e_ie_j=e_je_i$,
\[
e_je_ig_j=e_ig_je_i=g_je_ie_j.
\]
For $(i,j)=(2,1)$ we find the relation
$
e_2e_1=e_2g_2e_1{g_2}^{-1}=g_2e_1{g_2}^{-1}e_2,
$
i.e., since $g_2e_2=e_2g_2$,
\[
e_2e_1=e_2g_2e_1{g_2}^{-1}=g_2e_1e_2{g_2}^{-1},
\]
i.e., since $e_1e_2=e_2e_1$,
\[
e_1e_2{g_2}=e_2g_2e_1=g_2e_1e_2.
\]
Summing up, we see that relations \eqref{rBi}--\eqref{rBf} imply relations
\begin{align}
\label{rrBi} g_ig_j&=g_jg_i \quad\text{if}\quad |i-j|>1\\
g_{1}g_{2}g_{1}g_{2}&=g_{2}g_{1}g_{2}g_{1}\\
 g_{i}g_{j}g_{i}&=g_{j}g_{i}g_{j} \quad\text{if}\quad |i-j|=1, \{i,j\}\neq\{1,2\} \\
{g_i}^2 &= 1+(u_i-1)e_i(1-g_i)\\
g_ie_j&=e_jg_i\qquad \text{if}\quad |i- j|\neq 1 \quad \text{or}\quad (i,j)=(1,2)\\
 g_jg_ie_j&=e_ig_jg_i \qquad \text{if}\quad |i-j|= 1, \quad \{i,j\}\neq\{1,2\} \\
e_ie_j&=e_je_i\\
{e_i}^2&=e_i\\
 e_je_ig_j&=e_ig_je_i=g_je_ie_j\qquad \text{if}\quad |i-j|= 1, \quad \{i,j\}\neq\{1,2\}\\
 e_1e_2{g_2}&=e_2g_2e_1=g_2e_1e_2\\
 f_1&=e_1\\
 g_if_j&=f_jg_i \quad\text{if}\quad  i=1 \quad\text{or}\quad i\geq 2, j\neq i-1,i\\
 e_if_j&=f_je_i\\
 f_if_j&=f_jf_i\\
g_if_{i-1}&=f_{i}g_i\quad 2\leq i\leq n\\
 g_if_{i}&=f_{i-1}g_i\quad 2\leq i\leq n\\
e_{i+1}f_i&=e_{i+1}f_{i+1}\quad 1\leq i< n\\
 \label{rrBf} e_{i+1}f_i&=f_{i+1}f_{i}\quad 1\leq i< n.
\end{align}
The algebra $\mathcal E^B_n$ generated by $\{g_i,e_i,f_i\}_{i=1,\dots,n}$ subject to the relations \eqref{rrBi}-\eqref{rrBf} has been considered by Flores in \cite{F}, and called the {\it type $B$ braids and ties algebra}. See Remark \ref{rem:rescale} for a direct comparison of the relation  ${g_i}^2 = 1+(u_i-1)e_i(1-g_i)$ given above with the relation ${g_i}^2=1+({q_s}-{q_s}^{-1})e_sg_s$ in \cite{F}. The algebra $\mathcal E^B_n$ admits a geometric realization, provided in \cite{F}, in the spirit of Aicardi-Juyumaya braids and ties algebra. Hence we can repeat the argument given in Example \ref{ex:bt-a-to-A} to prove that relations \eqref{rrBi}-\eqref{rrBf} actually imply \eqref{rBi}-\eqref{rBf}. In particular,  
$\mathcal{E}^{(W,S)}(R[\Sigma(\Phi_U)_{[e]}],\rho_\phi,\hat{e})$ is isomorphic to  $\mathcal E^B_n$.

\end{example}
\section{Outlook}
The construction  of the algebras $\mathcal C_W$ has been upgraded to the case of complex reflection groups $W$ in \cite[Section 5]{marin} and  further investigated in  \cite{marin2}. It would be interesting to relate our constructions to this upgraded definition.\par
We believe that our abstract approach might be useful in handling constructions on ``braids and ties''-type monoids 
and algebras,  like the framization and deframization processes treated in \cite{AYP}. More in detail, let us recall the relationship between the algebra $\mathcal E_n$ from Example \ref{ex:bt-a-to-A} and the Yokonuma-Hecke algebra.  
In \cite{Y} Yokonuma provided a presentation by generators and relations for  the Hecke ring $H(G, U)$ where $G$ is a split simple Lie group of
adjoint type over $\mathbb F_q$, $T$ is a fixed maximal torus and  $U$ the unipotent radical of a fixed Borel subgroup containing $T$. Subsequently, Juyumaya and collaborators (\cite{JJ, JKK}) gave a different presentation which has been successfully related to diagram algebras, as follows. Consider  algebra $\mathrm{Y}_{d,n}(q)$  presented  with    invertible generators
 $g_1,\ldots ,g_{n-1}$, $t_1,\ldots ,t_n$ satisfying the following relations:
\begin{align}
\label{yk1}
 t_i t_j  = & t_j t_i,\quad
 t_i^d=  1,\\
\label{yk2}
t_j g_i  =  & g_i t_{s_i(j)}, \\
\label{yk3}
 g_i g_j  =  g_j g_i\quad \text{for $\vert i - j\vert > 1$}, &\quad\text{and}
\quad  g_i g_j g_i =  g_j g_ig_j\quad \text{for $\vert i - j\vert = 1$},\\
\label{yk4}
g_i^2 = & 1  +  (q-q^{-1}) \bar{e}_ig_i,
\end{align}
where
\begin{equation}\label{expei}
\bar{e_i}:= \frac{1}{d}\sum_{k=0}^{d-1}t_i^k t_{i+1}^{-k}.
\end{equation}
It turns out that the algebra $\mathcal E_n$   can be constructed  by removing  the framing generators $t_i$ of  the Yokonuma-Hecke algebra and by  considering a  new abstract generator $e_i$  satisfying the  same  relations  as the $\bar{e}_i$   in the Yokonuma-Hecke algebra. This process has been  studied  at a more abstract level of generality in \cite{AYP}, and it would be interesting  to investigate whether all of the  algebras $\mathcal{E}^{(W,S)}(R[\Sigma(\Phi_U)_{[e]}],\rho_\phi,\hat{e})$ can be obtained  in general by a similar deframization process.

\noindent {\bf Acknowledgements.} We would like to thank Jes\'us Juyumaya and the anonymous referee  for useful comments on the paper.

\end{document}